\documentclass{amsart}
\usepackage[utf8]{inputenc}
\usepackage{setspace}
\usepackage[margin=1.25in]{geometry}
\usepackage{graphicx}
\graphicspath{ {./figures/} }
\usepackage{subcaption}
\usepackage{amsmath}
\usepackage{amssymb}
\usepackage{mathtools}
\usepackage{lineno}
\usepackage{amsfonts}
\usepackage{xfrac} 
\usepackage{graphicx} 
\usepackage{amsthm}
\usepackage[all]{xy}

\usepackage[pagebackref=true]{hyperref}
\usepackage{float}
\usepackage{braket}

\usepackage{multirow}
\usepackage{diagbox}
\usepackage{slashbox}

\usepackage{tikz-cd}
\usetikzlibrary{quotes}
\usetikzlibrary{cd}

\newtheorem{theorem}{Theorem}

\newtheorem{example}[theorem]{Example}
\newtheorem{definition}[theorem]{Definition}

\newtheorem{lemma}[theorem]{Lemma}
\newtheorem{corollary}[theorem]{Corollary}
\newtheorem{remark}[theorem]{Remark}
\newtheorem{proposition}[theorem]{Proposition}

\newtheorem*{proposition*}{Proposition}
\newtheorem*{definition*}{Definition}
\newtheorem*{theorem*}{Theorem}

\begin{document}

\title{A Decomposition Theorem for Topological Branched Coverings}

\author{Shahryar Ghaed Sharaf}

\date{\today}



\begin{abstract}
In the context of complex algebraic varieties, the decomposition theorem for semi-small maps provides a decomposition of the direct image of the constant sheaf. In this work, we develop a decomposition theorem for branched coverings of topological spaces. To achieve this, we start by constructing a decomposition theorem for unramified covering maps via an explicit gluing construction using transition functions. For a given branched covering of closed topological manifolds, we use the previous result to establish a decomposition of the direct image of the constant sheaf on the covering space. In the next step, we generalize our discussion to the case where the target space is not necessarily a topological manifold.
\end{abstract}

\maketitle

\tableofcontents
	
\section{Introduction}
       For a proper map of complex algebraic varieties $f:X \longrightarrow Y$, a non-canonical decomposition of the direct image of the intersection complex $Rf_{\ast} IC_{X}(\underline{\mathbb{Q}})$ in the bounded derived category $D^{b}_{c}(Y)$ is given by the BBDG decomposition theorem. If the map $f:X \longrightarrow Y$ is a smooth projective morphism of complex algebraic manifolds, then a decomposition of the direct image $Rf_{\ast} \underline{\mathbb{Q}}_{X}$ is given by Deligne's decomposition theorem.
    	\begin{theorem*}[Deligne's Decomposition Theorem]
        For a smooth projective map of complex algebraic manifolds	$f:X \longrightarrow Y$, we have the following isomorphism in the bounded derived category $D_{c}^{b}(Y)$.
         \begin{align*}
	     Rf_{\ast}\underline{\mathbb{Q}}_{X} \simeq \bigoplus_{i \geq 0}\mathcal{H}^{i}(Rf_{\ast} \underline{\mathbb{Q}}_{X})[-i]
         \end{align*}	
         and the local systems $R^{i}f_{\ast}\underline{\mathbb{Q}}_{X}=\mathcal{H}^{i}(Rf_{\ast} \underline{\mathbb{Q}}_{X})$ are semi-simple.
     	\end{theorem*} 
    	A detailed proof of the above theorem is given by Deligne in \cite{deligne1968} and \cite[Theorem 4.2.6]{deligne1971}.
    	For a proper surjective morphism $f:X \longrightarrow Y$ with $X$ non-singular, we fix a stratification of $Y= \bigsqcup_{\alpha} S_{\alpha}$ with respect to which $f$ is stratified. The morphism $f$ is called semi-small if for every $\alpha$ and any point $s_{\alpha} \in S_{\alpha}$ the inequality $\dim_{\mathbb{C}}(S_{\alpha})+2\dim_{\mathbb{C}}(f^{-1}(s_{\alpha})) \leq \dim_{\mathbb{C}}(X)$ is satisfied. The decomposition theorem for semi-small morphisms provides a canonical decomposition of $Rf_{\ast}(\underline{\mathbb{Q}}_{X})$ in $\operatorname{Perv}(Y)$. A stratum $S_{\alpha}$ is called relevant if the equality $\dim_{\mathbb{C}}(S_{\alpha})+2\dim_{\mathbb{C}}(f^{-1}(s_{\alpha})) = \dim_{\mathbb{C}}(X)$ holds. For a relevant stratum $S_{\alpha}$, we define the local system $\mathcal{L}_{S_{\alpha} }:= (R^{n-\dim_{\mathbb{C}}S_{\alpha}}f_{\ast} \underline{\mathbb{Q}}_{X} )\vert_{S_{\alpha} }$. A detailed analysis of semi-small maps can be found in Borho and MacPherson's work \cite{BorhoPartialResolution}.
    	\begin{theorem*}[Decomposition Theorem for Semi-small Morphisms]
    	Choose a stratification of the complex algebraic variety $Y= \bigsqcup_{\alpha} S_{\alpha}$ with respect to which the proper surjective semi-small morphism $f:X \longrightarrow Y$ with $X$ non-singular is stratified. Let $\mathcal{S}_{\operatorname{rel}}$ denote the set of all relevant strata. Then, there is a canonical isomorphism in $\operatorname{Perv}(Y)$:
	\begin{align*}
		Rf_{\ast}\underline{\mathbb{Q}}_{X}[n] \simeq \bigoplus_{S_{\alpha} \in \mathcal{S}_{ \operatorname{rel} } }IC_{\overline{S}_{\alpha}  }(\mathcal{L}_{S_{\alpha} }  ),
	\end{align*}
	where $IC_{\overline{S}_{\alpha}  }(\mathcal{L}_{S_{\alpha} }  )$ is Deligne's sheaf complex regarding the middle perversity.	
	\end{theorem*}
	For a comprehensive treatment of the decomposition theorem for semi-small maps, see \cite{deCataldoHardLefschetz}. The aforementioned decomposition theorems concern algebraic varieties and algebraic morphisms. As is well-known, every algebraic morphism of complex algebraic varieties can be stratified (see, e.g., \cite{VerdierStratification}). Consider a stratified branched covering $ f:X \longrightarrow Y $ of complex algebraic varieties. Because the fiber $ f^{-1}(y) $ is zero-dimensional for every $ y \in Y $, the only relevant stratum is the top stratum $ S_{\alpha_{\text{top}}} $. The decomposition theorem for semi-small maps therefore yields a quasi-isomorphism
	$Rf_{\ast} \underline{\mathbb{Q}}_{X}[n] \simeq IC_{\overline{S}_{\alpha_{\text{top}}}}(\mathcal{L}_{{S_{\alpha_{\text{top}}}}})$.

	Let $f:M \longrightarrow N$ be a smooth proper submersion of smooth manifolds. Ehresmann’s theorem for smooth manifolds states that the map $f$ is a locally trivial smooth fiber bundle and that all fibers are diffeomorphic. The immediate cohomological consequence of Ehresmann’s theorem is that, for each $k$, the direct image sheaf $R^{k}f_{\ast} \underline{\mathbb{Q}}_{M}$ is a local system. Note that it also follows from Ehresmann's theorem that a smooth proper surjective submersion between closed manifolds of the same dimension is an unramified covering map. Hence, in this specific case—even though $M$ and $N$ are generally smooth manifolds rather than complex algebraic manifolds—the decomposition of $Rf_{\ast} \underline{\mathbb{Q}}_{M}$ follows directly from covering space theory. In general, however, while Ehresmann’s theorem fully describes the local topological behavior of a smooth proper submersion, Deligne's decomposition theorem applies only to smooth proper morphisms of complex algebraic manifolds. An example of a map where the extension of the conclusion of Deligne's decomposition theorem fails although Ehresmann’s theorem holds, is provided by the Hopf map $h:S^{3} \longrightarrow S^{2}$. Hence, even if the local topological behavior of a map is fully understood, there is no necessity of obtaining a decomposition theorem based solely on that local behavior.
	
	This work addresses the following question: To what extent can the algebraic decomposition theorems applied to branched coverings of complex algebraic varieties be generalized to a purely topological setting?

	To be more precise, we define topological branched coverings, the main objects of study in this work, as follows.

	 \begin{definition*}[Topological Branched Coverings]
	 	Let $X$ and $Y$ be path‑connected, locally compact Hausdorff spaces. A continuous, open, finite‑to‑one, proper surjection $f : X \longrightarrow Y$
		is called a \textbf{topological branched covering} if there exists a 
		closed nowhere dense subset $R \subset Y$ such that the restriction $f\big|_{f^{-1}(Y \setminus R)} : f^{-1}(Y \setminus R) \longrightarrow Y \setminus R$
		is a finite covering map. The set $Y \setminus R$ is called the \textbf{regular set}
		and $R$ is called the \textbf{branch locus} of $f$.
	\end{definition*}

	In \cite{fox1957covering}, Fox showed that a finite-to-one unbranched covering of a locally connected $T_1$ space $Y \setminus R$ can be uniquely extended to a branched covering $f:X \longrightarrow Y$. In this work, we do not assume any analytic or algebraic structure on $X$ or $Y$, but rather that they are topological manifolds and/or pseudomanifolds. Furthermore, we assume the branch locus $R \subset Y$ is a closed stratified space of codimension 2. However, it turns out that we still need a restriction on the inclusion $R \hookrightarrow Y$ in order to generalize the statements of the above decomposition theorems to the topological setting. Example \ref{non-example} provides an instance of a branched covering with the following property: for some point $r \in R$, there exists a small contractible open neighborhood 
	$U \subset Y$ of $r$ such that, for every open contractible $V \subseteq U$ containing $r$, the preimage $f^{-1}(V)$ is not contractible. The central idea of the proofs of the main results of this work, namely Theorem \ref{MainTheo1} and Proposition \ref{SingDecompProp1}, is to rule out the existence of such open neighborhoods. Definition \ref{localflatness} provides a setting that excludes the existence of branch loci with the property mentioned above, while preserving the topological setting without assuming any analytic or algebraic structure.
	
	In Section \ref{DecUnbranched}, we show that for an unramified covering map $f:X \longrightarrow Y$ of degree $n$, where $X$ and $Y$ are Hausdorff, path-connected, locally contractible, and locally compact topological spaces, we have, in $\operatorname{Sh}(Y)$, the isomorphism $f_{\ast} \underline{\mathbb{Q}}_{X} \cong \underline{\mathbb{Q}}_Y \oplus \mathcal{L}$, with $\mathcal{L}$ a local system of rank $n-1$. We further give an explicit description of $\mathcal{L}$ using the transition functions of $f$.

	The following theorem, which generalizes Deligne's decomposition theorem to branched coverings of topological manifolds, is the main result of Section \ref{DecompBranchedMfds}. However, since for a topological pseudomanifold it is possible to have strata of odd codimension, we prove the statement of the theorem for both the lower and the upper middle perversities.
	\begin{theorem*}[Theorem \ref{MainTheo1}]
		Let $f:X \longrightarrow Y$ be a branched covering between closed topological $n$-manifolds. Let the closed topologically stratified space $R \subset Y$ be the branch locus, with the stratification $R=R_{n-2} \supset R_{n-3} \supset \cdots$ and set $B:= f^{-1}(R)$. Denote by $\mathcal{L}$ the local system on $Y\setminus R$ associated to the unbranched covering $f|_{X\setminus B}: X\setminus B \longrightarrow Y\setminus R$. Then the direct image sheaf complex $Rf_{\ast}\underline{\mathbb{Q}}_{X} \in D^{b}_{c}(Y)$ can be decomposed as
		\begin{align*}
			Rf_{\ast}\underline{\mathbb{Q}}_{X}[n] \simeq \underline{\mathbb{Q}}_{Y}[n] \oplus IC_{Y}(\mathcal{L}),
		\end{align*}
		where Deligne's sheaf complex $IC_{Y}(\mathcal{L})$ is computed with respect to the stratification $Y \supset R_{n-2} \supset R_{n-3} \supset \cdots$, taking the perversity to be either the lower middle perversity or the upper middle perversity. Furthermore, the constructible sheaf $R^{0}j_{\ast}\mathcal{L}$ is uniquely determined (up to isomorphism) by the restriction map $f\vert_{X \setminus B}$, where $j:Y \setminus R \hookrightarrow Y$.
	\end{theorem*}
	
	We further generalize the above theorem to branched coverings with possibly singular target spaces. However, for this purpose we need a refinement of the intrinsic stratification of the pseudomanifold $Y$, which is provided by the following Proposition.
	\begin{proposition*}[Proposition \ref{RefinedStrat}]
		Let $f:X \longrightarrow Y$ be a branched covering of topological spaces such that $X$ is a closed topological manifold and $Y$ is a closed topological pseudomanifold. Furthermore, assume that the branch locus $R \subset Y$ is a closed stratified space of codimension 2 with the intrinsic stratification $R= \bigsqcup_{\beta}T_{\beta}$. Then, for the intrinsic stratification of $Y$, namely $Y= \bigsqcup_{\alpha} S_{\alpha}$, the decomposition 
		\begin{align*}
			Y=\big(S_{\alpha_{\text{top}}}\setminus (R \cap S_{\alpha_{\text{top}}}) \big) \bigsqcup_{\alpha, \beta } (S_{\alpha} \cap T_{\beta})
		\end{align*}
		is a refined stratification of $Y$, where $S_{\alpha_{\text{top}}}$ is the top stratum of $Y$.
	\end{proposition*}
	The following theorem is a generalization of the decomposition theorem for semi-small morphisms to branched coverings of topological spaces with possibly singular target spaces. Furthermore, since a topological pseudomanifold may have strata of odd codimension, we prove the statement of the theorem for both the lower and upper middle perversities.
	\begin{theorem*}[Theorem \ref{MainTheo2}]
		Let $f:X \longrightarrow Y$ be a branched covering, where $X$ is a closed topological manifold of dimension $n$, the topological space $Y$ is a closed topological pseudomanifold, and the branch locus is a closed topologically stratified subspace $R \subset Y$ of codimension 2. We consider the refined stratification obtained in the above  proposition. Then, in $D^{b}_{c}(Y)$, the following isomorphism holds:
		\begin{align*}
			Rf_{\ast}{\underline{\mathbb{Q}}}_{X}[n] \simeq IC_{Y}(\underline{\mathbb{Q}}_{Y \setminus R}) \oplus IC_{Y}(\mathcal{L}_{Y \setminus R}),
		\end{align*}
		where $\mathcal{L}_{Y \setminus R}$ is the local system associated with the unbranched covering $f\vert_{X \setminus f^{-1}(R)}: X \setminus f^{-1}(R) \longrightarrow Y \setminus R$ and Deligne's sheaf complexes $IC_{Y}(-)$ are computed with respect to the refined stratification, using either the lower or upper middle perversity.
	\end{theorem*}

    \section{Preliminaries}\label{Preliminaries}
    The first section of this work fixes notation and provides the necessary definitions, beginning with the central object of study: branched covering maps.

    \begin{definition}
    Let $X$ and $Y$ be path‑connected, locally compact Hausdorff spaces. 
    A continuous,open, finite‑to‑one, proper surjection $f : X \longrightarrow Y$
    is called a \textbf{topological branched covering} if there exists a 
    closed nowhere dense subset $R \subset Y$ such that the restriction $f\big|_{f^{-1}(Y \setminus R)} : f^{-1}(Y \setminus R) \longrightarrow Y \setminus R$
    is a finite covering map. The set $Y \setminus R$ is called the \textbf{regular set}
    and $R$ is called the \textbf{branch locus} of $f$.
\end{definition}

\begin{remark}
Note that for a given branched covering $f:X \longrightarrow Y$, the branch locus $R \subset Y$ is not unique. Any nowhere dense closed set $R^{\prime}$ that contains $R$ is also a branch locus. In this work, by the branch locus $R$ we mean the minimal branch locus, i.e., the set of all points $y \in Y$ that have no evenly covered neighborhood.
\end{remark}
As shown for example in \cite{fox1957covering}, the branch locus $R$ of a topological branched covering $f:X \longrightarrow Y$ is not, in general, a topologically stratified space. For the purposes of this work, however, it is sufficient to restrict our study to cases where the branch locus is a closed stratified space of codimension 2.

\begin{definition}\label{pseudomanifolds}
	We define a \textbf{topologically stratified space} inductively on dimension. A 0-dimensional topologically stratified space $X$ is a countable set with the discrete topology. For $m > 0$, an \textbf{$m$-dimensional topologically stratified space} is a para-compact Hausdorff topological space $X$ equipped with a filtration
	\begin{align*}
		X=X_{m} \supseteq X_{m-1} \supseteq \dots \supseteq X_{1} \supseteq X_{0} \supseteq X_{-1}= \emptyset
	\end{align*}
	by closed subsets $X_{j}$ such that if $x \in X_{j}-X_{j-1}$ there exists a neighborhood $\mathcal{N}_{x}$ of $x$ in $X$, a compact $(m-j-1)$-dimensional topologically stratified space $\mathcal{L}$ with filtration
	\begin{align*}
		\mathcal{L}=\mathcal{L}_{m-j-1} \supseteq \dots \supseteq \mathcal{L}_{1} \supseteq  \mathcal{L}_{0} \supseteq \mathcal{L}_{-1} = \emptyset,
	\end{align*}  
	and a homeomorphism $\phi : \mathcal{N}_{x} \longrightarrow \mathbb{R}^{j} \times C(\mathcal{L}),$
	where $C(\mathcal{L})$ is the open cone on $\mathcal{L}$, such that $\phi$ takes $\mathcal{N}_{x} \cap X_{j+i+1}$ homeomorphically onto
	\begin{align*}
		\mathbb{R}^{j} \times C(\mathcal{L}_{i}) \subseteq \mathbb{R}^{j} \times C(\mathcal{L})
	\end{align*}
	for $m-j-1 \geq i \geq 0$, and $\phi$ takes $\mathcal{N}_{x} \cap X_{j}$ homeomorphically onto
	\begin{align*}
		\mathbb{R}^{j} \times \{ \text{vertex of }\; C(\mathcal{L}) \}.
	\end{align*}
\end{definition}
\begin{remark}\label{linkandstratum}
	It follows that $X_{j}-X_{j-1}$ is a $j$-dimensional topological manifold. (The empty set is a manifold of any dimension.) We call the connected components of these manifolds the \textbf{strata} of $X$. Any $\mathcal{L}$ that satisfies the above properties is referred to as a \textbf{link} of the stratum at $x$.   
\end{remark}
\begin{definition}\label{PSMFD}
	An \textbf{$m$-dimensional topological pseudomanifold} is a para-compact Hausdorff topological space $X$ which possesses a topological stratification such that $X_{m-1}=X_{m-2}$
	and $X-X_{m-1}$ is dense in $X$.
\end{definition}

\begin{definition}[Goresky-MacPherson.]	
	A \textbf{perversity} 
	\begin{align*}
		\bar{p}: \mathbb{Z}_{ \geq 2} \longrightarrow \mathbb{Z}
	\end{align*}
	is a function such that
	$\overline{p}(2)=0$ and
	$\overline{p}(k+1)- \overline{p}(k) \in \{1,0\}$.
	The \textbf{complementary perversity} $\bar{q}$ of $\bar{p}$ 
	is the one with $\overline{p}(k)+\overline{q}(k)=k-2$. Furthermore, the lower middle perversity is defined by $\bar{m}(k)=[\frac{k-2}{2}]$. Its complementary perversity is called the upper middle perversity and is defined by $\bar{n}(k)=[\frac{k-1}{2}]$. 
\end{definition}

Throughout this work, the term "stratification" for a stratified space or a topological pseudomanifold refers to the intrinsic stratification. To be more precise, let $X$ be a stratified space. One begins by defining the top stratum as the largest open submanifold of $X$ and proceeds inductively based on the dimension of the strata. For details on this construction, the reader may consult \cite[V]{borel}.

\begin{definition}
	Let $X$ be a topological space. A \textbf{local system} (of abelian groups) on $X$ is a locally constant sheaf of abelian groups on $X$. In other words, a sheaf  $\mathcal{L}$ is a local system if every point has an open neighborhood $U$ such that the restricted sheaf $\mathcal{L}\vert_{U}$ is isomorphic to the sheafification of some constant presheaf.
\end{definition}
A central object of study in this work is Deligne's sheaf complex, which is defined as follows.
\begin{definition}\label{DeSheaf}
		Let $X^{n}$ be an $n$-dimensional pseudomanifold with a filtration
	$X = X_{n} \supseteq X_{n-2} \supseteq \cdots \supseteq X_{0} \supseteq \emptyset$, and set
	$U_{k} = X - X_{n-k}$ for $k \geq 2$, with inclusion maps $i_{k}:U_{k} \hookrightarrow U_{k+1}$. Let $\mathcal{L}$ be a local system on $X-X_{n-2}$. Deligne's sheaf complex $IC^{\overline{p}}_{X}(\mathcal{L})$ with respect to the perversity $\overline{p}$ is defined as 
	\begin{align*}
		IC^{\overline{p}}_{X}(\mathcal{L}) = \tau_{\leq \overline{p}(n)-n} Ri_{n \ast}  \cdots \tau_{\leq \overline{p}(3)-n} Ri_{3\ast}   \tau_{\leq \overline{p}(2)-n} Ri_{2 \ast} \mathcal{L}[n].
	\end{align*} 
\end{definition}

	In this paper, we denote Deligne’s sheaf complex by $IC^{\overline{p}}_{X}(-)$ and omit the perversity from the notation whenever it refers to the lower or the upper middle perversity. 

    \begin{remark}
    Verdier duality states that if $X$ is an orientable topological pseudomanifold of pure dimension $n$, then the Verdier dual of Deligne's sheaf complex associated to a perversity $\bar{p}$ and to the constant sheaf on the non-singular part of $X$ is quasi-isomorphic to Deligne's sheaf complex associated to the complementary perversity $\bar{q}$ and to the constant sheaf on the non-singular part of $X$, up to a shift by $n$. In particular, the lower and the upper middle perversities are of interest. It is known that complex algebraic varieties can be stratified such that the strata are smooth algebraic varieties. It follows that the strata are of even codimension, and hence, in this case, the lower and the upper middle perversities agree. But in general, for topological pseudomanifolds, it is possible to have strata of odd codimension. Hence, in this work, we distinguish between the lower and the upper middle perversities and prove the decomposition theorems \ref{MainTheo1} and \ref{MainTheo2} for both perversities.
    \end{remark}

    If the space $X$ is path-connected, a local system $\mathcal{L}$ of abelian groups has the same stalk $L$ at every point. The following proposition is an important result about such local systems for this work.
\begin{proposition}
	Let $X$ be a para-compact, Hausdorff, path-connected and	locally simply connected space. There is a bijective correspondence between local systems on $X$ and group homomorphisms
	$\rho :\pi _{1}(X,x) \longrightarrow \operatorname{Aut}(L)$.
\end{proposition}

Even if we rule out branched coverings whose branch locus is not a topologically stratified space, in a purely topological setting, there exist branched coverings $f: X \longrightarrow Y$ such that the branch locus $R \subset Y$ is a topologically stratified space of codimension $2$ and some points $r \in R$ have the following property: there exists a small contractible open neighborhood $ U \subset Y $ of $ r $ such that for every open contractible $ V \subseteq U $ containing $ r $, the preimage $ f^{-1}(V) $ is not contractible. We provide such an example below.

\begin{example}\label{non-example}
	Consider the Fox--Artin arc $A \subset S^3$ introduced in \cite{fox1948some} by Fox and Artin, a wild arc in the $3$-sphere. There exists a branched covering map $f: S^3 \to S^3$ with branch locus $R = A$.  The existence of such a branched covering follows from \cite{fox1957covering}, which shows that a finite-to-one unbranched covering of a locally connected $T_1$ space can be uniquely extended to a branched cover, and from the fact that the complement of a wild arc can have nontrivial fundamental group. Let $r \in R$ be a branch point that is also a wild point, and let $U \subset S^3$ be a sufficiently small open contractible neighborhood of $r$ (for instance, a topological $3$-ball). Due to the wild embedding of $A$ in $S^3$, for any open, contractible neighborhood $V \subseteq U$ containing $r$, the preimage $f^{-1}(V)$ fails to be contractible. More specifically, if $V$ is chosen to be a neighborhood of $r$ 
	that intersects $A$ in a non-trivial manner (which is unavoidable because $r$ lies on the wild arc), 
	then $f^{-1}(V)$ has a nontrivial fundamental group. This phenomenon occurs because the branched covering, 
	when restricted to such neighborhoods, is not equivalent to a standard cyclic cover branched along a locally flat submanifold; instead, the wildness of $A$ introduces non-trivial topology in the preimage. In other words, the wildness of the arc prevents the existence of locally trivializing neighborhoods for the branched covering in the usual sense.
\end{example}

However, this is not the case in the PL or Diff categories. In fact, local models in these categories rule out the existence of non-contractible preimages resulting from the wild embeddings mentioned above. For further details and an explicit study of these local models, the reader may consult, e.g., \cite{PiergalliniBranched} by Piergallini. Hence, we restrict our study to the following class of branched coverings. The goal is to maintain a pure topological setting while excluding wild embeddings such as the previous example.

\begin{definition}\label{localflatness}
	Let $Y = \bigsqcup_{\alpha} S_{\alpha}$ and $R = \bigsqcup_{\beta} T_{\beta}$ be topologically stratified spaces such that $R \subseteq Y$. An inclusion $R \hookrightarrow Y$ is said to be locally flat at a point $r \in R$ if the following holds. Suppose $r \in S_{\alpha} \cap T_{\beta}$ and $\dim(S_{\alpha}) \leq \dim(T_{\beta})$. Then there exist distinguished neighborhoods $U \subset R$ and $U^{\prime} \subset Y$ of $r$ such that $U \subset U^{\prime}$ and the topological pair $(U \cap T_{\beta},\, U^{\prime} \cap S_{\alpha})$ is homeomorphic to the topological pair $(\mathbb{R}^{\dim T_{\beta}},\, \mathbb{R}^{\dim S_{\alpha}})$ with the standard inclusion $\mathbb{R}^{\dim S_{\alpha}} \hookrightarrow \mathbb{R}^{\dim T_{\beta}}$. In other words, there exist homeomorphisms $U^{\prime} \cap S_{\alpha} \longrightarrow \mathbb{R}^{\dim S_{\alpha}}$ and $U \cap T_{\beta} \longrightarrow \mathbb{R}^{\dim T_{\beta}} $ making the following diagram commute:
	\begin{align*}		\begin{tikzcd}[ampersand replacement=\&]
			U^{\prime} \cap S_{\alpha} \arrow[r,  hookrightarrow] \arrow[d] \& U \cap T_{\beta} \arrow[d]   \\
			\mathbb{R}^{\dim S_{\alpha}} \arrow[r, hookrightarrow]  \& \mathbb{R}^{\dim T_{\beta} }  
		\end{tikzcd}.
	\end{align*}
	The inclusion $R \hookrightarrow Y$ is locally flat if it is locally flat at each point $r \in R$. If $\dim(S_{\alpha}) \geq \dim(T_{\beta})$, we define local flatness similarly.
\end{definition}

\begin{definition}\label{BranchLocFlat}
	A branched covering $f:X \longrightarrow Y$ of topological pseudomanifolds is called locally flat if the inclusions $R \hookrightarrow Y$ and $f^{-1}(R) \hookrightarrow X$ are locally flat, where $R \subset Y$ is the branch locus.
\end{definition}
In the following sections, the term ``branched coverings'' refers to locally flat branched coverings unless otherwise stated.
 
\section{A Decomposition Theorem for Topological Unramified Covering Maps}\label{DecUnbranched}
We start this section by fixing some notations. Let $f:X \longrightarrow Y$ be a covering map of Hausdorff, path-connected, locally contractible 
and locally compact topological spaces of degree $n$. Let $\{ U_{i} \}_{i \in I}$ be a good cover of $Y$; that is, for each $i \in I$, we have the homeomorphisms $\phi_{i}: f^{-1}(U_{i}) \xrightarrow{\cong} U_{i} \times F$ and $F \cong f^{-1}(y_{0})$ for a fixed base point $y_{0} \in Y$. Let $\cap_{ij}:= U_{i} \cap U_{j}$ be non-empty for $i \neq j$. Then the transition functions are $h_{ij}: \cap_{ij} \longrightarrow \operatorname{Aut}(F) \cong S_{n}$, where $S_{n}$ is the symmetric group of a set of $n$ elements. The local trivializing homeomorphisms $\phi_{i}$ and $\phi_{j}$ give rise to the homeomorphism $g_{ji}:=\phi_{j}\vert \circ \phi_{i}^{-1} \vert: \cap_{ij} \times F \xrightarrow{\cong} \cap_{ij} \times F$. Note that since $f^{-1}(y_{0})$ is a disjoint union of points, we obtain $R^{i}f_{\ast} \underline{\mathbb{Q}}_{X} \cong 0$ for $i \geq 1$. We aim to construct the sheaf $R^{0}f_{\ast}\underline{\mathbb{Q}}_{X}$ using the above topological data, inductively. The following proposition provides a proper starting point.
\begin{proposition}\label{PropDecomp1}
	Let $f:X \longrightarrow Y$ be a covering map of Hausdorff, path-connected, locally contractible 
	and locally compact topological spaces of degree $n$. Then in $\operatorname{Sh}(Y)$, we have 
	\begin{align}\label{Decomp1}
		f_{\ast} \underline{\mathbb{Q}}_{X} \cong \underline{\mathbb{Q}}_Y \oplus \mathcal{L},
	\end{align}
	where $\mathcal{L}$ is a local system of rank $n-1$.
\end{proposition}
\begin{proof}
	Consider the unit morphism under the $(f^{\ast},f_{\ast})$ adjunction $\underline{\mathbb{Q}}_{Y}\longrightarrow f_{\ast} f^{\ast} \underline{\mathbb{Q}}_{Y}$. For a surjective finite map, the unit morphism is injective. Recall that $f^{\ast} \underline{\mathbb{Q}}_{Y} \cong \underline{\mathbb{Q}}_{X}$ in $\operatorname{Sh}(X)$. Hence, we get an injective morphism of sheaves, $\eta : \underline{\mathbb{Q}}_{Y} \longrightarrow f_{\ast}\underline{\mathbb{Q}}_{X} $. Note that since the covering map is unramified, the sheaf $f_{\ast} \underline{\mathbb{Q}}_{X}$ is a local system of rank $n$. Let $\{U_{i}\}_{i \in I}$ be a good cover of $Y$. Then, the morphism $\eta_{U}: \Gamma(U;\underline{\mathbb{Q}}_{Y}) \longrightarrow \Gamma(U; f_{\ast} \underline{\mathbb{Q}}_{X})$ is given by $\eta_{U}(q)=(q, \dots ,q)$ for $U \in \{U_{i}\}_{i \in I}$. We also consider the counit morphism $\varepsilon:f_{\ast}f^{\ast} \underline{\mathbb{Q}}_{Y}(\cong f_{\ast} \underline{\mathbb{Q}}_{X}) \longrightarrow \underline{\mathbb{Q}}_{Y}$. The section morphism $\varepsilon_{U}:\Gamma(U;f_{\ast}\underline{\mathbb{Q}}_{X}) \longrightarrow \Gamma(U; \underline{\mathbb{Q}}_{Y})$ is given by $\varepsilon_{U}(q_{1}, \dots, q_{n})= \sum_{i=1}^{n}q_{i}$. It follows immediately that $\varepsilon_{U} \circ \eta_{U} = n \cdot \operatorname{id}_{\Gamma(U;\underline{\mathbb{Q}}_{Y})}$. As a result, the short exact sequence of sheaves
	\begin{align*}
		0 \rightarrow \operatorname{ker}(\varepsilon) \rightarrow f_{\ast} \underline{\mathbb{Q}}_{X} \xrightarrow{\varepsilon} \underline{\mathbb{Q}}_{Y} \rightarrow 0,
	\end{align*}  
	splits and the claim follows.
\end{proof}
\begin{remark}
A more general discussion of the unit (restriction) and counit (trace) morphisms for covering spaces can be found in \cite[VII.4]{iversen2012cohomology} by Iversen.
\end{remark}
The above proposition yields a decomposition of the direct image sheaf $f_{\ast}\underline{\mathbb{Q}}_{X}$. However, we aim to use the transition functions to give an explicit description of the sheaf $f_{\ast}\underline{\mathbb{Q}}_{X}$. Based on this, we will also provide an alternative proof of the above proposition, which leads to a concrete description of the local system $\mathcal{L}$.

Let $\{U_{l} \mid l \in I\}$ be a good cover of $Y$ and let $U_{i} \in \{U_{l} \mid l \in I\}$. We start with the natural isomorphism of sheaves $\underline{\mathbb{Q}}_{X} \vert_{\phi_{i}^{-1}(U_{i} \times F)} \cong \underline{\mathbb{Q}}_{f^{-1}(U_{i} )}$ and $(f \vert_{f^{-1}(U_{i} ) })_{\ast} \underline{\mathbb{Q}}_{f^{-1}(U_{i})} \cong \bigoplus_{l=1}^{n} \underline{\mathbb{Q}}_{U_{i}}$. Given $U_{i}$ and $U_{j}$ in $\{ U_{l} \vert l \in I \}$ such that $\cap_{ij}:= U_{i} \cap U_{j} \neq \emptyset$, the homeomorphism $g_{ij}: \cap_{ij} \times F \longrightarrow \cap_{ij} \times F$ induces an isomorphism of sheaves $g_{\ast}: \bigoplus_{l=1}^{n} \underline{\mathbb{Q}}_{\cap_{ij}} \big( \cong ( f \vert_{f^{-1}(\cap_{ij})})_{\ast} \underline{\mathbb{Q}}_{f^{-1}(\cap_{ij})}  \big) \longrightarrow \bigoplus_{l=1}^{n} \underline{\mathbb{Q}}_{\cap_{ij}}$. Consider the inclusions $\iota_{ij}: \cap_{ij} \hookrightarrow U_{i} \cup U_{j}$, $\iota_{i}: U_{i} \hookrightarrow U_{i} \cup U_{j}$, and $\iota_{j}: U_{j} \hookrightarrow U_{i} \cup U_{j}$. We extend the sheaves $\bigoplus_{l=1}^{n} \underline{\mathbb{Q}}_{\cap_{ij}}$, $\bigoplus_{l=1}^{n} \underline{\mathbb{Q}}_{U_{i}}$ and $\bigoplus_{l=1}^{n} \underline{\mathbb{Q}}_{U_{j}}$ by zero using the above inclusions to obtain sheaves in $\operatorname{Sh}(U_{i} \cup U_{j})$, namely $(\iota_{ij})_{!}(\bigoplus_{l=1}^{n} \underline{\mathbb{Q}}_{\cap_{ij}})$, $(\iota_{i})_{!}(\bigoplus_{l=1}^{n} \underline{\mathbb{Q}}_{U_{i}})$, and $(\iota_{j})_{!}(\bigoplus_{l=1}^{n} \underline{\mathbb{Q}}_{U_{j}})$ respectively. Let $k_{j}:\cap_{ij} \hookrightarrow U_{j}$ and $k_{i}: \cap_{ij} \hookrightarrow U_{i}$. It follows that $\iota_{ij}=\iota_{i} \circ k_{i}=\iota_{j} \circ k_{j}$, and hence $(\iota_{ij})_{!}=(\iota_{i})_{!} \circ (k_{i})_{!}=(\iota_{j})_{!} \circ (k_{j})_{!}$. As a result, we obtain the morphism $\alpha: (\iota_{ij})_{!}(\bigoplus_{l=1}^{n}\underline{\mathbb{Q}}_{\cap_{ij}}) \longrightarrow (\iota_{i})_{!}(\bigoplus_{l=1}^{n} \underline{\mathbb{Q}}_{U_{i}})$. Using the isomorphism $g_{\ast}:\bigoplus_{l=1}^{n}\underline{\mathbb{Q}}_{\cap_{ij}} \longrightarrow \bigoplus_{l=1}^{n}\underline{\mathbb{Q}}_{\cap_{ij}}$ and the inclusion $k_{j}:\cap_{ij} \hookrightarrow U_{j}$, we get the morphism $\beta: (\iota_{ij})_{!}(\bigoplus_{l=1}^{n}\underline{\mathbb{Q}}_{\cap_{ij}}) \xrightarrow{\simeq} (\iota_{ij})_{!}(\bigoplus_{l=1}^{n}\underline{\mathbb{Q}}_{\cap_{ij}}) \longrightarrow (\iota_{j})_{!}(\bigoplus_{l=1}^{n}\underline{\mathbb{Q}}_{U_{j}})$. We glue the sheaves $(\iota_{j})_{!}(\bigoplus_{l=1}^{n}\underline{\mathbb{Q}}_{U_{j}})$ and $(\iota_{i})_{!}(\bigoplus_{l=1}^{n}\underline{\mathbb{Q}}_{U_{i}})$ along $(\iota_{ij})_{!}(\bigoplus_{l=1}^{n}\underline{\mathbb{Q}}_{\cap_{ij}})$ using the morphisms $\alpha$ and $\beta$. Hence, we arrive at the following pushout diagram
\begin{align*}
	\begin{tikzcd}[ampersand replacement=\&]
		(\iota_{ij})_{!}(\bigoplus_{l=1}^{n}\underline{\mathbb{Q}}_{\cap_{ij}}) \arrow[r, "\beta"] \arrow[d, "\alpha"'] \& (\iota_{j})_{!}(\bigoplus_{l=1}^{n}\underline{\mathbb{Q}}_{U_{j}}) \arrow[d,] \\
		(\iota_{i})_{!}(\bigoplus_{l=1}^{n}\underline{\mathbb{Q}}_{U_{i}}) \arrow[r,] \& \mathcal{D}
	\end{tikzcd}.
\end{align*}
The following lemma ensures that the sheaf $\mathcal{D}$ is isomorphic to the restriction of the direct image sheaf.
\begin{lemma}
	Let $\mathcal{D}$ be the sheaf constructed above. Then, we have $\mathcal{D} \cong (f \vert_{f^{-1}(U_{i} \cup U_{j})})_{\ast}\underline{\mathbb{Q}}_{f^{-1}(U_{i} \cup U_{j})}$ in $\operatorname{Sh}(U_{i} \cup U_{j})$.
\end{lemma}
\begin{proof}
Consider the diagram
		\begin{align*}		\begin{tikzcd}[ampersand replacement=\&]
		(\iota_{ij})_{!}(\bigoplus_{l=1}^{n}\underline{\mathbb{Q}}_{\cap_{ij}}) \arrow[r, "\beta"] \arrow[d, "\alpha"'] \& (\iota_{j})_{!}(\bigoplus_{l=1}^{n}\underline{\mathbb{Q}}_{U_{j}}) \arrow[d,] \arrow[ddr,bend left=15, "\eta_{j}"] \& \\
		(\iota_{i})_{!}(\bigoplus_{l=1}^{n}\underline{\mathbb{Q}}_{U_{i}}) \arrow[r,] \arrow[rrd,bend right=15,"\eta_{i}" '] \& \mathcal{D} \arrow[dr, dashed, "\exists! "] \& \\ \& \& (f \vert_{f^{-1}(U_{i} \cup U_{j})})_{\ast}\underline{\mathbb{Q}}_{f^{-1}(U_{i} \cup U_{j})}
	\end{tikzcd},
	\end{align*}
	where we define the morphisms $\eta_{i}$ and $\eta_{j}$ as follows. For $U \subseteq U_{j}$ connected, the section homomorphism $\Gamma(U;(\iota_{j})_{!}(\bigoplus_{l=1}^{n}\underline{\mathbb{Q}}_{U_{j}})) (\cong \mathbb{Q}^{n}) \longrightarrow \Gamma(U;(f \vert_{f^{-1}(U_{i} \cup U_{j})})_{\ast}\underline{\mathbb{Q}}_{f^{-1}(U_{i} \cup U_{j})}) (\cong \mathbb{Q}^{n})$ is the identity. Otherwise, if $U \nsubseteq U_{j}$ is connected, then  $\Gamma(U;(\iota_{j})_{!}(\bigoplus_{l=1}^{n}\underline{\mathbb{Q}}_{U_{j}}))=0$ and the induced homomorphisms on sections are the trivial homomorphisms. On the other hand, for an open connected $U \subseteq U_{i}$, the homomorphism $\Gamma(U;(\iota_{i})_{!}(\bigoplus_{l=1}^{n}\underline{\mathbb{Q}}_{U_{i}})) (\cong \mathbb{Q}^{n}) \longrightarrow \Gamma(U;(f \vert_{f^{-1}(U_{i} \cup U_{j})})_{\ast}\underline{\mathbb{Q}}_{f^{-1}(U_{i} \cup U_{j})}) (\cong \mathbb{Q}^{n})$ is defined to be the isomorphism induced by $g_{\ast}:\bigoplus_{l=1}^{n}\underline{\mathbb{Q}}_{\cap_{ij}} \longrightarrow \bigoplus_{l=1}^{n}\underline{\mathbb{Q}}_{\cap_{ij}}$ on the stalks. Otherwise, if $U \nsubseteq U_{i}$ and $U$ is connected the induced homomorphisms on sections are the trivial homomorphisms. One can easily check that the above section homomorphisms define sheaf morphisms $\eta_{i}$ and $\eta_{j}$. The commutativity of the outer diagram follows directly from the definitions. By the universal property of pushouts, there exists a unique morphism $\mathcal{D} \longrightarrow  (f \vert_{f^{-1}(U_{i} \cup U_{j})})_{\ast}\underline{\mathbb{Q}}_{f^{-1}(U_{i} \cup U_{j})}$, as indicated in the diagram. It is straightforward to prove that the obtained morphism induces isomorphisms on stalks; the claim follows by considering the three cases $y \in \cap_{ij}$, $y \in (U_{i} \setminus \cap_{ij})$, and $y \in (U_{j} \setminus \cap_{ij})$.
\end{proof}
The above lemma can be applied inductively to construct the direct image sheaf $f_{\ast}\underline{\mathbb{Q}}_{X}$.

\begin{proposition}\label{fvert}
		Let $f:X \longrightarrow Y$ be a covering map of Hausdorff, path-connected, locally contractible 
		and locally compact topological spaces of degree $n$. There is a morphism $\varepsilon: f_{\ast}\underline{\mathbb{Q}}_{X} \longrightarrow \underline{\mathbb{Q}}_{Y}$ such that $\mathcal{L} \cong \operatorname{ker}(\varepsilon)$, where the local system $\mathcal{L}$ is defined in Equation \ref{Decomp1}.
\end{proposition}
\begin{remark}
	The above proposition follows directly from the splitting of the short exact sequence $0 \rightarrow \operatorname{ker}(\varepsilon) \rightarrow f_{\ast} \underline{\mathbb{Q}}_{X} \xrightarrow{\varepsilon} \underline{\mathbb{Q}}_{Y} \rightarrow 0$ presented in the proof of Proposition \ref{PropDecomp1}. However, our goal here is to provide an alternative proof using the permutation matrices $\sigma \in S_{n}$ to demonstrate the existence of the morphism $\varepsilon:f_{\ast} \underline{\mathbb{Q}}_{X} \longrightarrow \underline{\mathbb{Q}}_{Y}$.  
\end{remark}
\begin{proof}
	Let $\{U_{i}\}_{i \in I}$ be a good cover of $Y$. Let the open sets $U_{i}, U_{j} \in \{U_{l}\}_{l \in I}$ have a non-empty intersection. We define $\cap_{ij}:= U_{i} \cap U_{j} $. Consider the pushout diagram
		\begin{align*}		\begin{tikzcd}[ampersand replacement=\&]
			(\iota_{ij})_{!}(\bigoplus_{l=1}^{n}\underline{\mathbb{Q}}_{\cap_{ij}}) \arrow[r, "\beta"] \arrow[d, "\alpha"'] \& (\iota_{j})_{!}(\bigoplus_{l=1}^{n}\underline{\mathbb{Q}}_{U_{j}}) \arrow[d,] \arrow[ddr,bend left=15, "\varepsilon_{j}"] \& \\
			(\iota_{i})_{!}(\bigoplus_{l=1}^{n}\underline{\mathbb{Q}}_{U_{i}}) \arrow[r,] \arrow[rrd,bend right=15,"\varepsilon_{i}" '] \& (f \vert_{f^{-1}(U_{i} \cup U_{j})})_{\ast}\underline{\mathbb{Q}}_{f^{-1}(U_{i} \cup U_{j})} \arrow[dr, dashed, "\exists! "] \& \\ \& \& \underline{\mathbb{Q}}_{U_{i} \cup U_{j}}
		\end{tikzcd},
	\end{align*}
	where the morphisms $\alpha$ and $\beta$ are defined above. We define the morphism $\varepsilon_{i}$ as follows. For an open connected $U \subseteq U_{i}$, the homomorphism of sections $\varepsilon_{i}: \Gamma(U; (\iota_{i})_{!}(\bigoplus_{l=1}^{n} \underline{ \mathbb{Q}}_{U_{i}}  ) ) (\cong \mathbb{Q}^{n}) \longrightarrow \Gamma(U;\underline{\mathbb{Q}}_{U_{i} \cup U_{j}}) (\cong \mathbb{Q})$ is defined to be $\varepsilon_{i}(s_{1}, \dots , s_{n})= \sum_{l=1}^{n}s_{l}$. Otherwise, if $U \nsubseteq U_{i}$ and $U$ is connected the homomorphism of sections $\varepsilon_{i}: \Gamma(U; (\iota_{i})_{!}(\bigoplus_{l=1}^{n} \underline{ \mathbb{Q}}_{U_{i}}  )) (= 0) \longrightarrow \Gamma(U;\underline{\mathbb{Q}}_{U_{i} \cup U_{j}}) (\cong \mathbb{Q})$ is the trivial homomorphism. It can be easily checked that $\varepsilon_{i}$ is a well-defined morphism of sheaves. We define the morphism $\varepsilon_{j}$, similarly. The commutativity of the outer diagram follows from the following consideration. Recall that the transition functions are given by $g_{ij}: \cap_{ij} \longrightarrow \operatorname{Aut}(F) \cong S_{n}$, where $F = f^{-1}(y_{0})$ for a fixed $y_{0} \in Y$. This topological data is encoded in the morphism $\beta$. To be more precise, for an open connected $U \subseteq \cap_{ij}$, the homomorphism $\Gamma(U;(\iota_{ij})_{!}(\bigoplus_{l=1}^{n}\underline{\mathbb{Q}}_{\cap_{ij}}))( \cong \mathbb{Q}^{n}) \longrightarrow \Gamma(U;(\iota_{j})_{!}(\bigoplus_{l=1}^{n}\underline{\mathbb{Q}}_{U_{j}}))(\cong \mathbb{Q}^{n})$ is given by a permutation matrix $\sigma \in S_{n} \subset \operatorname{GL}_{n}(\mathbb{Q})$. But notice that $\sum_{l=1}^{n}s_{l} = \sum_{l=1}^{n}s_{\sigma(l)}$ for all $(s_{1}, \dots , s_{n}) \in \mathbb{Q}^{n}$. By the universal property of pushouts, there is a unique morphism of sheaves $\varepsilon_{ij}: (f \vert_{f^{-1}(U_{i} \cup U_{j})})_{\ast}\underline{\mathbb{Q}}_{f^{-1}(U_{i} \cup U_{j})} \longrightarrow \underline{\mathbb{Q}}_{U_{i} \cup U_{j}}$ that makes the above diagram commutative. The morphism $\varepsilon_{ij}$ is surjective on each stalk and hence surjective. The morphism $\varepsilon: f_{\ast}\underline{\mathbb{Q}}_{X} \longrightarrow \underline{\mathbb{Q}}_{Y}$ can be constructed inductively, using the above procedure. From Equality \ref{Decomp1} it follows that $\operatorname{ker}(\varepsilon) \cong \mathcal{L}$.
\end{proof}

\begin{remark}
	Recall that the assignment $U \longmapsto \operatorname{ker}\big( (\Gamma(U;f_{\ast} \underline{\mathbb{Q}} _{X}) \xrightarrow{\varepsilon_{U}} \Gamma(U; \underline{\mathbb{Q}}_{Y})\big)$ defines a sheaf. Hence, for a good cover of $Y$, namely $\{U_{l}\}_{l \in I}$, we obtain
	\begin{align*}
		\Gamma(U_{i}; \mathcal{L})= \{ (s_{1}, \dots , s_{n}) \in \Gamma(U_{i};f_{\ast}\underline{\mathbb{Q}}_{X})\cong \bigoplus_{l=1}^{n} \Gamma(U_{i};\underline{\mathbb{Q}}_{Y}) \mid \sum_{l=1}^{n}s_{l}=0\}  \; \text{for any} \; U_{i} \in \{U_{l}\}_{l \in I} 
	\end{align*}
\end{remark}

\section{A Decomposition Theorem for Topological Branched Coverings}\label{DecompBranchedMfds}
In this section, we extend our discussion to the case of branched coverings. Let $f:X \longrightarrow Y$ be a branched covering of Hausdorff, path-connected, locally contractible 
and locally compact topological spaces. Let $R \subset Y$ be the branch locus, and $B:= f^{-1}(R) \subset X$. The main goal of this section is to show that the derived direct image sheaf $Rf_{\ast} \underline{\mathbb{Q}}_{X}$ decomposes as the direct sum of the constant sheaf $\underline{\mathbb{Q}}_{Y}$ and a constructible sheaf on $Y$. We will also provide an explicit construction of the constructible sheaf appearing in the decomposition. Consider the following commutative diagram
	\begin{align*}		\begin{tikzcd}[ampersand replacement=\&]
		X \setminus B \arrow[r,  hookrightarrow,"i"] \arrow[d, "f \vert"'] \& X \arrow[d,"f"]   \\
		Y \setminus R \arrow[r, hookrightarrow,"j"]  \& Y   
	\end{tikzcd},
\end{align*}
Using the result from the previous section, we obtain $j_{\ast} \circ (f\vert)_{\ast}\underline{\mathbb{Q}}_{X \setminus B} \cong j_{\ast}(\underline{\mathbb{Q}}_{Y \setminus R} \oplus \mathcal{L} ) \cong j_{\ast} (\underline{\mathbb{Q}}_{Y \setminus R}) \oplus j_{\ast}\mathcal{L}$, where the local system $\mathcal{L}$ was constructed in the previous section. By the functoriality of the direct image, we then have $ f_{\ast} \circ i_{\ast} \underline{\mathbb{Q}}_{X \setminus B}  \cong j_{\ast} (\underline{\mathbb{Q}}_{Y \setminus R}) \oplus j_{\ast}\mathcal{L}$. The direct images of the constant sheaves on dense open subspaces under the inclusions $i$ and $j$ do not always yield the constant sheaf on the target space. However, we will give a sufficient condition under which the sheaves $j_{\ast} \underline{\mathbb{Q}}_{Y \setminus R}$ and $i_{\ast} \underline{\mathbb{Q}}_{X \setminus B}$ can be expressed as a direct sum of the constant sheaves and constructible sheaves on the respective target spaces. We begin by showing that one can always find monomorphisms $\underline{\mathbb{Q}}_{X} \longrightarrow i_{\ast}\underline{\mathbb{Q}}_{X \setminus B}$ and $\underline{\mathbb{Q}}_{Y} \longrightarrow j_{\ast}\underline{\mathbb{Q}}_{Y \setminus R}$.
\begin{lemma}\label{ShExSeq}
	Let $f:X \longrightarrow Y$ be a branched covering of Hausdorff, path-connected, locally contractible 
	and locally compact topological spaces. Let $R \subset Y$ denote the branch locus of $f$ and set $B:= f^{-1}(R)$. Then there are short exact sequences of sheaves
	\begin{align*}
		0 \longrightarrow \underline{\mathbb{Q}}_{X} \longrightarrow i_{\ast} \underline{\mathbb{Q}}_{X \setminus B} \longrightarrow \operatorname{Coker}(\underline{\mathbb{Q}}_{X} \rightarrow i_{\ast} \underline{\mathbb{Q}}_{X \setminus B}) \longrightarrow 0, \\
		0 \longrightarrow \underline{\mathbb{Q}}_{Y} \longrightarrow j_{\ast} \underline{\mathbb{Q}}_{Y \setminus R} \longrightarrow \operatorname{Coker}(\underline{\mathbb{Q}}_{Y} \rightarrow j_{\ast} \underline{\mathbb{Q}}_{Y \setminus R}) \longrightarrow 0.
	\end{align*}
\end{lemma}

\begin{proof}
	Consider the adjunction morphism $\underline{\mathbb{Q}}_{X} \longrightarrow i_{\ast} i^{\ast} \underline{\mathbb{Q}}_{X}$. Recall that the open set $U:=X \setminus B \subset X$ is dense. Hence we have the isomorphism $i^{\ast}\underline{\mathbb{Q}}_{X} \cong \underline{\mathbb{Q}}_{X \setminus B}$. For an open subset $V \subset X$ the section morphisms are given by the restriction morphisms
	$\Gamma(V; \underline{\mathbb{Q}}_{X}) \longrightarrow \Gamma(V \cap U; \underline{\mathbb{Q}}_{X})=\Gamma(V; i_{\ast}\underline{\mathbb{Q}}_{U})$. Since $U \subset X$ is dense, the restriction morphisms are injective and the claim follows. Showing the injectivity of the morphism $\underline{\mathbb{Q}}_{Y} \longrightarrow j_{\ast}\underline{\mathbb{Q}}_{Y \setminus R}$ goes along similar lines.
\end{proof}
Having established the above lemma, the question of whether the sheaves $i_{\ast} \underline{\mathbb{Q}}_{X \setminus B}$ and $j_{\ast} \underline{\mathbb{Q}}_{Y \setminus R}$ can be written as direct sums of the constant sheaves $\underline{\mathbb{Q}}_{X}$ and $\underline{\mathbb{Q}}_{Y}$ and some constructible sheaves boils down to whether the above short exact sequences split. As an immediate consequence of the above lemma, we obtain the following corollary.
\begin{corollary}
	In the situation of Lemma \ref{ShExSeq}, if the extension classes $[\eta_{X}] \in \operatorname{Ext}^{1}(\operatorname{Coker}(\underline{\mathbb{Q}}_{X} \rightarrow i_{\ast} \underline{\mathbb{Q}}_{X \setminus B}), \underline{\mathbb{Q}}_{X})$ and $[\eta_{Y}] \in \operatorname{Ext}^{1}(\operatorname{Coker}(\underline{\mathbb{Q}}_{Y} \rightarrow j_{\ast} \underline{\mathbb{Q}}_{Y \setminus R}), \underline{\mathbb{Q}}_{Y})$ of the respective short exact sequences vanish, then the short exact sequences split.
\end{corollary}  
   The above corollary gives an algebraic approach to the question of decomposing $i_{\ast} \underline{\mathbb{Q}}_{X \setminus B}$ and $j_{\ast} \underline{\mathbb{Q}}_{Y \setminus R}$. However, in what follows, we consider cases in which the morphisms $\underline{\mathbb{Q}}_{X} \longrightarrow i_{\ast} \underline{\mathbb{Q}}_{X \setminus B}$ and $\underline{\mathbb{Q}}_{Y} \longrightarrow j_{\ast} \underline{\mathbb{Q}}_{Y \setminus R}$ are in fact isomorphisms.
\subsection{Derived Direct Image Sheaf under Branched Coverings of Topological Manifolds}
The preceding discussion provides a general framework for the study of the direct image sheaf
 \begin{align*}
 	Rf_{\ast} \underline{\mathbb{Q}}_{X},
 \end{align*}
 where $f:X \longrightarrow Y$ is a branched covering of Hausdorff, path-connected, locally contractible 
 and locally compact topological spaces. It is a classical result that if $X$ and $Y$ are topological manifolds and the branch locus $R \subset Y$ is a closed topologically stratified space, then the real codimension of the branch locus is at least 2 (see, for example, \cite{chernavskii1964finite}, \cite{vaisala1967discrete}). Our arguments will reprove that the real codimension of the branch locus of ramified coverings between closed topological manifolds is in fact equal to 2. We start this subsection by showing that if the map $f:X \longrightarrow Y$ is a branched covering of topological manifolds, the short exact sequences described in Lemma \ref{ShExSeq} split. With this established, we proceed to formulate and prove a decomposition theorem for the direct image sheaf $Rf_{\ast}\underline{\mathbb{Q}}_{X}$.
 
 \begin{definition}
 	Let $X$ be a topological space and $x \in X$. A direct system of open neighborhoods of $x$ is a family
 	$\{V_l\}_{l \in I}$ of open subsets of $X$ with $x \in V_i$ for all $i$,
 	indexed by a directed set $(I,\le)$ (i.e.\ a non‑empty partially ordered set
 	in which any two indices have an upper bound) such that $i \le j \;\Longrightarrow\; V_i \supseteq V_j$.
 	If in addition every open neighbourhood of $x$ contains some $V_i$, the system is called cofinal (or a neighbourhood basis).
 \end{definition}

\begin{lemma}
    Let $U \subset X$ be an open dense subset of the locally path-connected space $X$, and let $i : U \hookrightarrow X$ be the inclusion. Assume that for each
	$x \in X \setminus U$ there is a directed system $\mathcal{N}_x$ of
	path-connected open neighborhoods of \(x\) which forms a neighborhood
	basis, and that for every $V \in \mathcal{N}_x$ the intersection
	$V \cap U$ is path-connected. Then
	\begin{align*}
	i_*\underline{\mathbb{Q}}_U \cong \underline{\mathbb{Q}}_X
	\quad \text{in } \operatorname{Sh}(X).
	\end{align*}
\end{lemma}
\begin{proof}
Consider the adjunction morphism $\underline{\mathbb{Q}}_{X} \longrightarrow i_{\ast}i^{\ast}\underline{\mathbb{Q}}_{X}$. Since $U \subset X$ is open and dense, we have an isomorphism $\underline{\mathbb{Q}}_{U} \cong i^{\ast}\underline{\mathbb{Q}}_{X}$. We prove that the morphism $\underline{\mathbb{Q}}_{X} \longrightarrow i_{\ast} \underline{\mathbb{Q}}_{U}$ induces isomorphisms on all stalks and therefore defines an isomorphism of sheaves. Since $X$ is locally path-connected, for each $x \in U$ we may choose an open path-connected neighborhood, namely $V$. The section morphism $\Gamma(V; \underline{\mathbb{Q}}_{X})(\cong \mathbb{Q}) \longrightarrow \Gamma(V; i_{\ast}\underline{\mathbb{Q}}_{U})(\cong \mathbb{Q})=\Gamma(V;\underline{\mathbb{Q}}_{X})\cong \mathbb{Q}$ is the identity. Consequently, since the existence of a directed system is ensured by the assumption and it is a neighborhood basis, the induced morphism on the stalks $(\underline{\mathbb{Q}}_{X})_{x}(\cong \mathbb{Q}) \longrightarrow (i_{\ast}\underline{\mathbb{Q}})_{x}(\cong \mathbb{Q})$ is also the identity. For $x \in X \setminus U$, let $V_{i}$ be an open neighborhood containing $x$. By the same argument the induced morphism on the section is the identity, which again implies that the morphism on the stalks is the identity.
\end{proof}
For a given branched covering $f:X \longrightarrow Y$ between closed topological manifolds, assuming that the branch locus $R \subset Y$ is a closed topologically stratified space, one can use the previous lemma to prove that the morphisms $\underline{\mathbb{Q}}_{X} \longrightarrow i_{\ast} \underline{\mathbb{Q}}_{X \setminus B} $ and $\underline{\mathbb{Q}}_{Y} \longrightarrow j_{\ast} \underline{\mathbb{Q}}_{Y \setminus R}$ described in Lemma \ref{ShExSeq} are in fact isomorphisms. This observation paves the way to providing a decomposition theorem for the direct image sheaf $Rf_{\ast} \underline{\mathbb{Q}}_{X}$.

\begin{proposition}\label{IsoIncl}
	Let $f:X \longrightarrow Y$ be a branched covering between closed topological manifolds. Let $R \subset Y$ denote the branch locus, which is a closed topologically stratified space of real codimension equal to or greater than 2. Then the morphisms
	 \begin{align*}
	\underline{\mathbb{Q}}_{X} \longrightarrow i_{\ast}(\underline{\mathbb{Q}}_{X \setminus B}) \; \text{and}\; \underline{\mathbb{Q}}_{Y}\longrightarrow j_{\ast}(\underline{\mathbb{Q}}_{Y \setminus R}),
	\end{align*}
	described in Lemma \ref{ShExSeq}, are isomorphisms.
\end{proposition}
	\begin{proof}
		We assume that the real codimension of $R \subset Y$ equals 2. For the cases where $R$ has greater codimension, the proof proceeds along similar lines. For each $y \in R$, choose a small open neighborhood $U$ of $y$ in $Y$ such that $U \cong \mathbb{R}^{n}$ and $U \cap R \cong \mathbb{R}^{j} \times C(L^{(n-2)-j-1})$, where $n= \dim(Y)$ and $L$ is a link of the point $y$ in $R$. The homotopy type of the space $U \setminus (U \cap R)$ will play an important role in our subsequent discussion. In fact, the homotopy type of $U \setminus (U \cap R)$ is entirely determined by the embedding $U \cap R \hookrightarrow U$. Note that from the local flatness assumption, it follows that the inclusion $\mathbb{R}^{j} \hookrightarrow \mathbb{R}^{n}$ is the standard inclusion. Consequently, we can write $\mathbb{R}^{n} \cong \mathbb{R}^{j} \times C( S^{n-j-1} )$ in such a way that the inclusion $ \mathbb{R}^{j} \times C(L) \hookrightarrow \mathbb{R}^{j} \times C( S^{n-j-1} )$, when restricted to the first factor, is a homeomorphism. Using the described embedding, we arrive at
		\begin{align*}
		U \setminus (U \cap R) \cong \mathbb{R}^{j} \times C(S^{n-j-1}) \setminus (\mathbb{R}^{j} \times C(L)) \cong \mathbb{R}^{j} \times (S^{n-j-1} \setminus L) \times (0,1]\simeq S^{n-j-1} \setminus L.
		\end{align*}
		The topological stratified space $L$ is compact and locally contractible, which allows us to use Alexander duality. Using Alexander duality, we obtain 
		\begin{align*}
			\widetilde{H}_{q}(S^{n-j-1} \setminus L ) \cong H^{(n-j-1)-q-1}(L).
		\end{align*}
		For $q=0$, we arrive at $\widetilde{H}_{0}(S^{n-j-1} \setminus L ) \cong H^{n-j-2}(L)=0$, since the topological pseudomanifold $L$ is $(n-j-3)$-dimensional. As a consequence, it follows that the topological space $U \setminus (U \cap R)$ is connected. Let $R=\bigsqcup_{\beta} T_{\beta}$ be the intrinsic stratification of $R$. Choosing $U$ sufficiently small such that $U \cap T_{\beta} \cong \mathbb{R}^{\dim(T_{\beta})}$ and considering that $U \cap T_{\beta} \hookrightarrow U$ is homeomorphic to the standard inclusion implies that the map $f \vert_{f^{-1}(U \cap T_{\beta})}:f^{-1}(U \cap T_{\beta}) \longrightarrow U \cap T_{\beta}$ is a homeomorphism on each connected component of $f^{-1}(U \cap T_{\beta})$. It follows that $B \subset X$ is of codimension 2. Since, by assumption, the inclusion $B \hookrightarrow X$ is locally flat in the sense of Definition \ref{localflatness}, it follows that for each $x \in B$, we can find an open neighborhood $V \cong \mathbb{R}^{n}$ in $X$ such that $V \setminus (V \cap B)$ is connected. For each $x$ and each $y$, choosing the open neighborhoods $U^{\prime} \subseteq U$ and $V^{\prime} \subseteq V$, such that $U^{\prime}$ and $V^{\prime}$ satisfy the above conditions as $U$ and $V$, respectively, we obtain directed systems of open neighborhoods. The claim then follows from the previous lemma.
	\end{proof}
	The following remark briefly describes how to associate a local system to a given representation of a fundamental group.
	\begin{remark}\label{LocalandRep}
		Let $V$ be a vector space. The equivalence between the categories of local systems with stalks isomorphic to $V$ and the category of representations $\rho: \pi_{1}(X,x_{0}) \longrightarrow \operatorname{GL}(V)$ is a well-known result. Let $U \subseteq X$ be an open subset. We aim to describe the group of sections $\Gamma(U; \mathcal{L}_{\rho})$ using the corresponding representation $\rho$. For simplicity, we assume that $X$ is path-connected. We first outline how a local system can be associated with a given representation $\rho:\pi_{1}(X,x_{0}) \longrightarrow \operatorname{GL}(V)$. We further assume that $X$ is locally simply connected. Consider the universal covering $p:\widetilde{X} \longrightarrow X$. Note that the universal cover $\widetilde{X}$ is constructed as
		\begin{align*}
			\widetilde{X}: \sfrac{ \{\gamma: \;\text{ $\gamma$ is a path in $X$ with $\gamma(0)=x_{0}$ }  \} }{ \sim},
		\end{align*} 
	where $\gamma \sim \gamma^{\prime}$ if and only if $\gamma(1) = \gamma^{\prime}(1) $ and $\gamma \simeq \gamma^{\prime}$. Then the universal covering is given by $p([\gamma])=\gamma(1)$. Using the preceding notation, the action of the fundamental group $\pi_{1}(X,x_{0})$ on $\widetilde{X}$ is given by
	\begin{align*}
	\pi_{1}(X,x_{0}) \times \widetilde{X} &\rightarrow \widetilde{X} \\
	([\eta], [\gamma] ) &\mapsto [\eta \ast \gamma].	
	\end{align*}
	We define the flat vector bundle $E_{\rho} = \sfrac{(\widetilde{X} \times V )}{\sim^{\prime}}$, where the equivalence relation $\sim^{\prime}$ is induced by the action of $\pi_{1}(X,x_{0})$ on $\widetilde{X} \times V$ given by
	\begin{align*}
	\pi_{1}(X,x_{0}) \times ( \widetilde{X} \times V) &\rightarrow \widetilde{X} \times V \\
	([\eta], [\gamma] , v ) &\mapsto ([\eta \ast \gamma], \rho([\eta]) \cdot v ).	
	\end{align*}
	The local system $\mathcal{L}_{\rho}$ is the sheaf of locally constant sections of $E_{\rho} \longrightarrow X$. More explicitly, let $U \subseteq X$ be an open subset. An element $s \in \Gamma(U, \mathcal{L}_{\rho} )$ corresponds to a $\pi_{1}$-equivariant locally constant function $f:p^{-1}(U) \longrightarrow V$ satisfying $f([\eta] [\gamma])= \rho([\eta]) \cdot f([\gamma])$ for each $[\eta] \in \pi_{1}(X,x_{0})$ and each $[\gamma] \in p^{-1}(U)$. Now, assume that $U \subseteq X$ is open and path-connected. Choose a base point $u \in U$ and a path $\alpha$ from $x_{0}$ to $u$. Let $i:U \hookrightarrow X$ be the inclusion and $i_{\ast}: \pi_{1}(U,u) \longrightarrow \pi_{1}(X,x_{0})$ the induced homomorphism between fundamental groups. Hence, we have
	\begin{align*}
		\Gamma(U; \mathcal{L}_{\rho})= \{v \in V \vert \rho([\eta])v=v \; \forall \; [\eta] \in \operatorname{Im}(\pi_{1}(U,u)\xrightarrow{i_{\ast}} \pi_{1}(X,x_{0})) \}.
	\end{align*} 
	\end{remark}
	We now prove a proposition that will be essential for the main theorem of this section.
	\begin{proposition}\label{ConstantIndex}
		Let $ f: X \longrightarrow Y $ be a branched covering between closed topological $n$-manifolds. 
		Let the branch locus $ R \subset Y $ be a closed topologically stratified space of codimension 2. 
		Then, for each stratum of $ R $, the cardinality of the fiber $ f^{-1}(y) $ is constant as $ y $ varies within the stratum.
	\end{proposition}
	\begin{proof}
	Let $B:=f^{-1}(R)$. Consider the following commutative diagram
	\begin{align*}		\begin{tikzcd}[ampersand replacement=\&]
			X \setminus B \arrow[r,  hookrightarrow,"i"] \arrow[d, "f \vert"'] \& X \arrow[d,"f"]   \\
			Y \setminus R \arrow[r, hookrightarrow,"j"]  \& Y   \\ 
		\end{tikzcd}.
	\end{align*}	
	It follows that we have the identity $f_{\ast} \circ i_{\ast} \underline{\mathbb{Q}}_{X \setminus B}  \cong j_{\ast} \circ (f \vert)_{\ast} \underline{\mathbb{Q}}_{X \setminus B} $ in $\operatorname{Sh}(Y)$. Using Propositions \ref{IsoIncl} and \ref{PropDecomp1}, we obtain $f_{\ast}\underline{\mathbb{Q}}_X \cong \underline{\mathbb{Q}}_Y \oplus j_{\ast}\mathcal{L}$, where the local system $\mathcal{L}$ is described in Equation \ref{Decomp1}. By the definition, we have $\operatorname{rk}((f_{\ast}\underline{\mathbb{Q}}_{X})_{y } )= \operatorname{rk}(\underset{y \in U}{\varinjlim}\, H^{0}(f^{-1}(U) ))=\operatorname{rk}(H^{0}(f^{-1}(y)))$ for a point $y \in R$. Using the previous identity, we arrive at $(f_{\ast} \underline{\mathbb{Q}}_X)_{y}= 1 + \operatorname{rk}((j_{\ast}\mathcal{L})_{y})$. As in the proof of Proposition \ref{IsoIncl}, for each $y \in R$ we choose a small open neighborhood $U$ in $Y$ such that $U \cong \mathbb{R}^{n}$ and $U \cap R \cong \mathbb{R}^{j} \times C(L^{n-2-j-1})$, where the topologically stratified space $L$ is a link of the point $y$ in $R$. Using Remark \ref{LocalandRep}, a local system on the space $U \setminus (U \cap R)$ is given by a representation $\rho: \pi_{1}(S^{n-j-1} \setminus L) \longrightarrow \operatorname{GL}_{k-1}(\mathbb{Q})$, where we used $U\setminus (U \cap R) \simeq S^{n-j-1} \setminus L$, as shown in the proof of Proposition \ref{IsoIncl}. Note that $\operatorname{rk}(\Gamma(U;j_{\ast}\mathcal{L}))= \operatorname{rk} \{v \in \mathbb{Q}^{k-1} \; \vert \; \rho(\gamma)v=v \;\forall \; \gamma \in \pi_{1}(S^{n-j-1} \setminus L) \}$, where $k$ is the index of the unbranched part of the covering, namely $f\vert_{X \setminus B}$. Finally, we arrive at
	\begin{align}\label{rkequ}
		\operatorname{rk}(f_{\ast}(\underline{\mathbb{Q}}_{X})_{y})=1+\operatorname{rk} \{v \in \mathbb{Q}^{k-1} \; \vert \; \rho(\gamma)v=v \;\forall \; \gamma \in \pi_{1}(S^{n-j-1} \setminus L) \}.
	\end{align}
	
	We still must show that for any two points $y, y^{\prime}$ in the same connected component of a stratum $T \subset R$, the dimensions of the invariant subspaces in Equation \ref{rkequ} coincide.  
	By local flatness, for each $y \in T$ we can choose a distinguished neighborhood $U \cong \mathbb{R}^j \times C(S^{n-j-1})$ such that  
	$U \cap R \cong \mathbb{R}^j \times C(L)$ and consequently  
	$U \setminus (U \cap R) \cong \mathbb{R}^j \times (S^{n-j-1} \setminus L) \times (0,1] \simeq S^{n-j-1} \setminus L$.  
	The link complement $S^{n-j-1} \setminus L$ is well-defined up to homotopy for all points in the same connected component of a stratum. The local system $\mathcal{L}\) on \(Y \setminus R$ restricts to a local system on $U \setminus (U \cap R)$.  
	Since $\mathbb{R}^j$ is contractible, any local system on $\mathbb{R}^j \times (S^{n-j-1} \setminus L)$ is isomorphic to the pullback of a local system on $S^{n-j-1} \setminus L$ via the projection $\mathbb{R}^j \times (S^{n-j-1} \setminus L) \longrightarrow S^{n-j-1} \setminus L$.  
	Hence the representation $\rho_y : \pi_1(S^{n-j-1} \setminus L) \longrightarrow \operatorname{GL}_{k-1}(\mathbb{Q})$ associated to $\mathcal{L}$ at $y$ is well-defined up to conjugation; in particular the dimension of the invariant subspace $\operatorname{rk}\{v \in \mathbb{Q}^{k-1} \mid \rho_y(\gamma)v = v \;\forall \gamma\}$ depends only on the conjugacy class of $\rho_y$. Let $y,y^{\prime} \in T$. Now take a path $\gamma : [0,1] \longrightarrow T$ connecting $y$ and $y^{\prime}$.  
	Cover the image of $\gamma$ by finitely many such neighborhoods $U_{i}$ with the product structure.  
	On each overlap $U_{i} \cap U_{i+1}$ the two local descriptions of $\mathcal{L}$ are related by the transition functions of $\mathcal{L}$ (which are isomorphisms of stalks). These transition functions induce isomorphisms between the fibers of $\mathcal{L}$ at nearby points, and therefore conjugate the corresponding representations of $\pi_1(S^{n-j-1} \setminus L)$.  
	Thus the representations $\rho_y$ and $\rho_{y'}$ are conjugate, and consequently the dimensions of their invariant subspaces are equal. This argument shows that the quantity  
	$1 + \operatorname{rk}\{v \in \mathbb{Q}^{k-1} \mid \rho(\gamma)v = v \;\forall \gamma \in \pi_1(S^{n-j-1} \setminus L)\}$ is constant on each connected component of $T$.  
	Therefore the cardinality of the fiber, $\operatorname{rk}(f_*\underline{\mathbb{Q}}_X)_y$, is constant on each connected component of $T$.
   
    It follows that, for a given connected component of a stratum, the rank of the stalks of the sheaf $f_{\ast}(\underline{\mathbb{Q}}_{X})$---and hence the cardinality of the fibers over branch points of the map $f:X \longrightarrow Y$---is uniquely determined by the representation of the local system $\mathcal{L}\vert_{U \setminus (U \cap R)}$ which does not change along a connected component of a stratum of $R$. 	Since the representation of the local system $\mathcal{L}\vert_{U \setminus (U \cap R)}$ does not depend on the chosen stratification of $R$, the above argument remains valid for any stratification of $R$. In fact, refining a stratification of $R$ may add more strata, but the cardinality of the fiber at a given point $y \in R$ does not change.
	\end{proof}
	The following example is a special case of the main theorem of this section.
	
	\begin{example}\label{ExampleMfd}
	Let $f:X \longrightarrow Y$ be a branched covering of topological $n$-manifolds such that the branch locus $R \subset Y$ is a submanifold of codimension 2. Let $B:= f^{-1}(R)$. The goal of this example is to give a decomposition of the complex of sheaves $Rf_{\ast} \underline{\mathbb{Q}}_{X}$. Note that from Proposition \ref{ConstantIndex}, it follows that the map $f\vert_{B}:B \longrightarrow R$ is an unramified covering. As shown in the proof of Proposition \ref{IsoIncl}, for each point $y \in R$ there is a small open neighborhood $U \subset Y$ such that $U \cong \mathbb{R}^{n}$, $U \cap R \cong \mathbb{R}^{n-2}$, $U \setminus (U \cap R) \cong \mathbb{R}^{n} \setminus \mathbb{R}^{n-2}   \simeq S^{1}$, and $f \vert_{f^{-1}(U \cap R)}$ is a trivial covering.

	Let $V \subseteq f^{-1}(U)$ be a connected component. Consider the restricted branched covering $f \vert_V :V \longrightarrow U$ with branch locus $U \cap R$. In other words, we choose $V$ such that for a branch point the cardinality of the fiber over the branch point equals one. Since we are studying branched locally flat coverings in the sense of Definition \ref{BranchLocFlat}, it is possible to choose $U$ sufficiently small such that the map $f \vert_{V \cap f^{-1}(U \cap R) }$ is a homeomorphism and the branching locus $U \cap R$ is contractible. Furthermore, since $U \cap R$ is contractible and $U \cap R \hookrightarrow U$ is the standard inclusion and hence a cofibration, we have $\sfrac{U}{(U \cap R)} \simeq U$. It follows that the map $f \vert_{\operatorname{quo.}}:\sfrac{V}{(V \cap f^{-1}(U \cap R))} \longrightarrow \sfrac{U}{(U \cap R)}$ is a branched covering, the branch locus is a single point and $V \simeq \sfrac{V}{(V \cap f^{-1}(U \cap R))}$. Note that $U \cap R(\cong \mathbb{R}^{n-2}) \hookrightarrow U(\cong \mathbb{R}^{n})$ is the standard inclusion, and hence $U \setminus (U \cap R) \cong (S^{n-1} \setminus S^{n-3}) \times (0,1]$, and $\partial \overline{U} \setminus (\partial (\overline{U}) \cap R) \cong S^{n-1} \setminus S^{n-3}$, with the convention $S^{-1} = \emptyset $. The restriction $f \vert_{(V \cap f^{-1}((\partial \overline{U} \setminus (\partial (\overline{U}) \cap R)) ))}:= f \vert_{\partial}$ is an unramified covering and the map $f \vert_{\operatorname{quo.}}$ is homeomorphic to the cone $C(f \vert_{\partial}):C(V \cap f^{-1}((\partial \overline{U} \setminus (\partial (\overline{U}) \cap R)))) \longrightarrow C(\partial \overline{U} \setminus (\partial (\overline{U}) \cap R))$ which implies $\sfrac{V}{(V \cap f^{-1}(U \cap R))} \cong C(V \cap f^{-1}((\partial \overline{U} \setminus (\partial (\overline{U}) \cap R)))$. As a result, we obtain the homotopy equivalence of the topological spaces $V \simeq \ast$. The previous consideration shows that for each point $r \in R$ there exists an open neighborhood $U( \cong \mathbb{R}^{n}) \subset Y$ such that for each contractible open neighborhood $W \subseteq U$ of $r$, the preimage $f^{-1}(W)$ is contractible. Consider the following commutative diagram
	\begin{align*}		\begin{tikzcd}[ampersand replacement=\&]
			X \setminus B \arrow[r,  hookrightarrow,"i"] \arrow[d, "f \vert"'] \& X \arrow[d,"f"]   \\
			Y \setminus R \arrow[r, hookrightarrow,"j"]  \& Y   \\ 
		\end{tikzcd}.
	\end{align*}
     
     From the commutativity of the diagram, it follows that $Rf_{\ast} \circ Ri_{\ast} \underline{\mathbb{Q}}_{X \setminus B} \simeq Rj_{\ast} \circ R(f \vert_{X \setminus B})_{\ast} \underline{\mathbb{Q}}_{X \setminus B}$. Recall that we have $R^{l}(f \vert_{X \setminus B})_{\ast}\underline{\mathbb{Q}}_{X \setminus B} \simeq 0$ for $l>0$ and $(f \vert_{X \setminus B})_{\ast} \underline{\mathbb{Q}}_{X \setminus B} \simeq \underline{\mathbb{Q}}_{Y \setminus R} \oplus \mathcal{L}$, where the local system $\mathcal{L}$ is defined in Equation \ref{Decomp1}. As a result, we obtain $Rf_{\ast} \circ Ri_{\ast} \underline{\mathbb{Q}}_{X \setminus B} \simeq Rj_{\ast} \underline{\mathbb{Q}}_{Y \setminus R} \oplus Rj_{\ast} \mathcal{L}$. Thus, we get $\tau_{\leq -n}Rf_{\ast} \circ Ri_{\ast} \underline{\mathbb{Q}}_{X \setminus B}[n] \simeq \tau_{\leq -n}Rj_{\ast} \underline{\mathbb{Q}}_{Y \setminus R}[n] \oplus \tau_{\leq -n} Rj_{\ast} \mathcal{L}[n]$.

    Consider the adjunction between the derived direct image $Rf_{\ast}$ and the derived inverse image $Lf^{\ast}$, together with the fact that $Lf^{\ast}$ commutes with truncation functors (since it is exact). Hence, for the truncation $\tau_{\leq -n}$, the counit of the adjunction gives a natural morphism
       
   \begin{align*}
   \tau_{\leq -n} Rf_{\ast} Ri_{\ast} \underline{\mathbb{Q}}_{X \setminus B}[n] \longrightarrow Rf_{\ast} \tau_{\leq -n} Ri_{\ast} \underline{\mathbb{Q}}_{X \setminus B}[n]
   \end{align*}
   Assume that we have shown that the filtrations $Y \supset R$ and $X \supset B$ are stratifications of $X$ and $Y$. It follows that $\tau_{\leq -n} Ri_{\ast} \underline{\mathbb{Q}}_{X \setminus B}[n] \simeq \underline{\mathbb{Q}}_{X}[n]$. The existence of sufficiently small contractible open neighborhoods $W$, as constructed above, such that $f^{-1}(W)$ is contractible as well, implies that $R^{l}f_{\ast}(\tau_{\leq -n} Ri_{\ast} \underline{\mathbb{Q}}_{X \setminus B}[n])_{y}=0$ for $l > 0$ and $y \in Y$. For $l=0$, since $f$ is a finite map, the above natural morphism induces an isomorphism on the stalks. In fact, the rank of the stalk at each point equals the cardinality of its preimage under the map 
   $f$. It follows that the induced morphism in the derived category is a quasi-isomorphism. 
   Hence, we arrive at $Rf_{\ast} \circ \tau_{\leq -n}Ri_{\ast} \underline{\mathbb{Q}}_{X \setminus B}[n] \simeq \tau_{\leq -n}Rj_{\ast} \underline{\mathbb{Q}}_{Y \setminus R}[n] \oplus \tau_{\leq -n}Rj_{\ast} \mathcal{L}[n]$. Since $X$ and $Y$ are manifolds, we obtain
     \begin{align*}
     Rf_{\ast}  \underline{\mathbb{Q}}_{X}[n] \simeq \underline{\mathbb{Q}}_{Y}[n] \oplus IC_{Y}(\mathcal{L}).
     \end{align*}
	\end{example}
	In the next step, we relax the assumption on the branch locus $R \subset Y$ in the above example and consider an arbitrary topologically stratified space $R$ as the branch locus.
	\begin{theorem}\label{MainTheo1}
		Let $f:X \longrightarrow Y$ be a branched covering between closed topological $n$-manifolds. Let the closed topologically stratified space $R \subset Y$ be the branch locus, with the stratification $R=R_{n-2} \supset R_{n-3} \supset \cdots$ and set $B:= f^{-1}(R)$. Then the direct image sheaf complex $Rf_{\ast}\underline{\mathbb{Q}}_{X} \in D^{b}_{c}(Y)$ can be decomposed as
		\begin{align}
			Rf_{\ast}\underline{\mathbb{Q}}_{X}[n] \simeq \underline{\mathbb{Q}}_{Y}[n] \oplus IC_{Y}(\mathcal{L}),
		\end{align}
		where the local system $\mathcal{L}$ is defined in Equation \ref{Decomp1} and the complex $IC_{Y}(\mathcal{L})$ is computed with respect to the stratification $Y \supset R_{n-2} \supset R_{n-3} \supset \cdots$. Furthermore, the constructible sheaf $R^{0}j_{\ast}\mathcal{L}$ is uniquely determined (up to isomorphism) by the restriction map $f\vert_{X \setminus B}$.
	\end{theorem}
	\begin{proof}
	We start by showing that the filtration $X \supset B \supset B_{n-3} \supset \cdots$ is a stratification of $X$, where $B_{i}:=f^{-1}(R_{i})$. Let $x \in B_{j} \setminus B_{j-1}$ such that $f(x)=y$. As in the proof of Proposition \ref{IsoIncl}, for the point $y \in R_{j} \setminus R_{j-1}$ we choose a small open neighborhood $U$ in $Y$ such that $U \cong \mathbb{R}^{n}$ and $U \cap R \cong \mathbb{R}^{j} \times C(L^{(n-2)-j-1})$, where the topologically stratified space $L$ is a link of the point $y$ in $R$. Furthermore, we choose $U$ sufficiently small such that the restriction map $f \vert_{f^{-1}(U \cap (R_{j} \setminus R_{j-1}))}$ is a trivial covering. Recall that Proposition \ref{ConstantIndex} implies that the restriction map $f \vert_{B_{j} \setminus B_{j-1}}:B_{j} \setminus B_{j-1} \longrightarrow R_{j} \setminus R_{j-1}$ is an unramified covering. It follows that the subspace $B_{j} \setminus B_{j-1} \subset X$ is a topological manifold. Consider the stratification of $Y=\bigsqcup_{j}(R_{j} \setminus R_{j-1})=\bigsqcup_{j}S_{j}$. Let $S_{j}$ and $S_{j^{\prime}}$ be strata of $Y$ such that $S_{j} \subset \overline{ S_{j^{\prime} } }$. Define $T_{j}:= f^{-1}(S_{j})$ and $T_{j^{\prime}}:=f^{-1}(S_{j^{\prime}})$ and suppose $T_{j} \cap \overline{ T_{j^{\prime}} } \neq \emptyset$. Then for any $x \in T_{j}$, we have $f(x) \in S_{j}$. Since $S_{j} \subset \overline{ S_{j^{\prime} } }$, every open neighborhood of $f(x)$ intersects $S_{j^{\prime}}$. Because $f$ is an open map, every open neighborhood of $x$ intersects $T_{j^{\prime}}$. Thus, $x \in \overline{T_{j^{\prime}}}$, so $T_{j} \subset \overline{T_{j^{\prime}}}$. Hence, the decomposition $X=\bigsqcup_{j} T_{j}$ is a stratification of $X$.\\
	In the next step, similar to Example \ref{ExampleMfd}, we want to show that for the branch point $y \in R_{j} \setminus R_{j-1}$ there exists an open neighborhood $W \subseteq U$ such that the preimage $f^{-1}(W)$ is contractible. We generalize the idea presented in the previously mentioned example. Let $V \subseteq f^{-1}(U)$ be a connected component. Consider the restricted branched covering $f \vert_{V}:V \longrightarrow U$ with branch locus $U \cap R$. Since we considered a connected component of $f^{-1}(U)$, it implies that for $U$ sufficiently small the map $f \vert_{V \cap f^{-1}(U \cap R)}$ is a homeomorphism. Note that $U \setminus (U \cap R) \cong \mathbb{R}^{n} \setminus (\mathbb{R}^{j} \times C(L)) \cong \mathbb{R}^{j} \times C(S^{n-j-1}) \setminus (\mathbb{R}^{j} \times C(L)) \cong \mathbb{R}^{j} \times (0,1] \times (S^{n-j-1} \setminus L)$, which implies that $\partial \overline{U} \setminus (\partial (\overline{U}) \cap R) \cong \{0\} \times D^{j} \times (S^{n-j-1} \setminus L )$. The restriction map $f \vert_{V \cap f^{-1}((\partial \overline{U} \setminus (\partial (\overline{U}) \cap R )))}:= f \vert_{\partial }$ is an unramified covering. Furthermore, the map $f \vert_{ \operatorname{quo.} }: \sfrac{V}{(V \cap f^{-1}(U \cap R))} \longrightarrow \sfrac{U}{(U \cap R)}$ is a branched covering with the branch locus being a single point, since the map $f \vert_{V \cap f^{-1}(U \cap R)}$ is a homeomorphism. As a result we have a homeomorphism of maps $C(f \vert_{\partial}) \cong f \vert_{\operatorname{quo.}}$. Note that $U \cap R \simeq \ast$ and $U \cap R \hookrightarrow U$ is a cofibration, hence the previous homeomorphism of maps implies $V \simeq \sfrac{V}{(V \cap f^{-1}(U \cap R))} \cong C\big(V \cap (\partial \overline{U} \setminus (\partial (\overline{U}) \cap R ))\big) \simeq \ast$.\\
	Finally, consider the commutative diagram	
    \begin{align*}		\begin{tikzcd}[ampersand replacement=\&]
    		X \setminus B \arrow[r,  hookrightarrow,"i_{2}"] \arrow[d, "f \vert_{2}"'] \& X \setminus B_{n-3}  \arrow[d,"f \vert_{3}"] \arrow[r, hookrightarrow,"i_{3}"] \& X \setminus B_{n-4}  \arrow[d,"f \vert_{4}"] \arrow[r, hookrightarrow,"i_{4}"]  \& \cdots \\
    		Y \setminus R \arrow[r, hookrightarrow,"j_{2}"]  \& Y \setminus R_{n-3} \arrow[r,hookrightarrow,"j_{3}"] \& Y \setminus R_{n-4} \arrow[r, hookrightarrow,"j_{4}"] \& \cdots  \\ 
    	\end{tikzcd}.
    \end{align*}	
   Since the covering map $f \vert_{2}$ is unramified, using the same notation as in Equality \ref{Decomp1}, we have $Rf \vert_{2 \ast} \underline{\mathbb{Q}}_{X \setminus B} \simeq \underline{\mathbb{Q}}_{Y \setminus R} \oplus \mathcal{L}$ in $D_{c}^{b}(Y \setminus R)$. Let $\overline{p}$ be either the lower or the upper middle perversity, and let $\tau_{\leq}$ be the truncation functor. Applying the operator $\tau_{\leq \overline{p}(n)-n} Rj_{n \ast} \cdots	\tau_{\leq \overline{p}(3)-n} Rj_{3 \ast} \tau_{\leq \overline{p}(2)-n} Rj_{2 \ast}$ to the previous quasi-isomorphism, we obtain
   \begin{align*}
   		&\tau_{\leq \overline{p}(n)-n} Rj_{n \ast} \cdots 	\tau_{\leq \overline{p}(3)-n} Rj_{3 \ast} \tau_{\leq \overline{p}(2)-n} Rj_{2 \ast}Rf \vert_{2 \ast} \underline{\mathbb{Q}}_{X \setminus B}[n] \\ &\simeq
   		\tau_{\leq \overline{p}(n)-n} Rj_{n \ast} \cdots 	\tau_{\leq \overline{p}(3)-n} Rj_{3 \ast} \tau_{\leq \overline{p}(2)-n} Rj_{2 \ast}(\underline{\mathbb{Q}}_{Y \setminus R}[n]  \oplus \mathcal{L}[n] ).
   \end{align*}
   The commutativity of the first left-hand square and the existence of sufficiently small contractible open neighborhoods with contractible preimages yield
   \begin{align*}
   \tau_{\leq \overline{p}(2)-n} Rj_{2 \ast}Rf \vert_{2 \ast} \underline{\mathbb{Q}}_{X \setminus B}[n]  \simeq \tau_{\leq \overline{p}(2)-n} Rf \vert_{3 \ast} Ri_{2 \ast} \underline{\mathbb{Q}}_{X \setminus B}[n] .
   \end{align*}
   
   Similar to Example \ref{ExampleMfd}, Consider the adjunction between the derived direct image $Rf\vert_{3\ast}$ and the derived inverse image $Lf\vert_{3}^{\ast}$, together with the fact that $Lf\vert_{3}^{\ast}$ commutes with truncation functors (since it is exact). Hence, for the truncation $\tau_{\leq -n}$, the counit of the adjunction gives a natural morphism
   
   \begin{align*}
   	\tau_{\leq \overline{p}(2) -n} Rf\vert_{3\ast} Ri_{2\ast} \underline{\mathbb{Q}}_{X \setminus B}[n] \longrightarrow Rf \vert_{3\ast} \tau_{\leq \overline{p}(2)-n} Ri_{2\ast} \underline{\mathbb{Q}}_{X \setminus B}[n]
   \end{align*}
   
    Note that, since $X$ is a manifold, we have $\tau_{\leq \overline{p}(2) -n} Ri_{2\ast} \underline{\mathbb{Q}}_{X \setminus B}[n] \simeq \underline{\mathbb{Q}}_{X \setminus B_{n-3}}[n]$. The existence of sufficiently small contractible open neighborhoods, as constructed above, such that the preimage under the map $f$ is contractible as well, implies that $R^{l}f_{\ast}(\tau_{\leq \overline{p}(2) -n} Ri_{2\ast} \underline{\mathbb{Q}}_{X \setminus B}[n])_{y}=0$ for $l > 0$ and $y \in Y$. For $l=0$, since $f$ is a finite map, the above natural morphism induces an isomorphism on the stalks. In fact, the rank of the stalk at each point equals the cardinality of its preimage under the map 
    $f$. It follows that the induced morphism in the derived category is a quasi-isomorphism. Thus, we have the quasi-isomorphism
   \begin{align*}
   \tau_{\leq \overline{p}(2)-n} Rf \vert_{3 \ast} Ri_{2 \ast} \underline{\mathbb{Q}}_{X \setminus B}[n] \simeq  Rf \vert_{3 \ast} \tau_{\leq \overline{p}(2)-n} Ri_{2 \ast} \underline{\mathbb{Q}}_{X \setminus B}[n].
   \end{align*}

    Let $\overline{p}$ be the upper or lower middle perversity, then the claim follows by induction and the quasi-isomorphism $IC_{X}(\underline{\mathbb{Q}}_{X \setminus B}) \simeq \underline{\mathbb{Q}}_{X}$.
   \end{proof}

\begin{corollary}
 In the setting of Theorem \ref{MainTheo1}, the codimension of the branch locus $R$ cannot be greater than $2$.
\end{corollary}
\begin{proof}
		Assume that the branch locus of the map $f:X \longrightarrow Y$ is a topological stratified space $R \subset Y$ of codimension at least $3$. Let $y \in R$ be a point in the branch locus $R$ and $n=\dim_{\mathbb{R}}(Y)$. Choose a small open neighborhood $U \subseteq Y$ of $y$ such that $U \cong \mathbb{R}^{n}$ and $U \cap R \cong \mathbb{R}^{j} \times C(L)$, where $L$ is a link of the point $y$. As shown in the proof of Proposition \ref{IsoIncl}, we have $U \setminus (U \cap R) \simeq S^{n-j-1} \setminus L$. Using dimension theory and transversality arguments, we deduce that removing a closed pseudomanifold $L$ of codimension at least $3$ from the sphere $S^{n-j-1}$ does not change its fundamental group. Hence, $\pi_{1}(S^{n-j-1} \setminus L)=0$. From Equality \ref{rkequ}, it follows that $\operatorname{rk}(Rf_{\ast}(\underline{\mathbb{Q}}_{X})_{y})=d$, where $d$ is the degree of the map $f \vert_{X \setminus f^{-1}(R)}$. This contradicts the assumption that $y$ is a branch point, since at a branch point the rank must be strictly less than $d$.
\end{proof}

\subsection{Direct Image Sheaf under Branched Coverings over Singular Target Spaces}
In this subsection, we develop a decomposition theorem for branched coverings $f:X \longrightarrow Y$ requiring only that the covering space $X$ is a closed topological manifold. We begin with the following simple observation.\\
\begin{proposition}
	Let $f:X \longrightarrow Y$ be a branched covering with $X$ a  closed topological manifold, $Y$ a closed topological pseudomanifold, and the branch locus $R \subset Y$ a closed topologically stratified space of codimension 2. Then $Y_{\text{sing}}$, the singular part of $Y$, satisfies $Y_{\text{sing}} \subseteq R$.
\end{proposition} 
\begin{proof}
	Assume that $Y_{\text{sing}} \setminus R \neq \emptyset$. Recall that $(Y \setminus R) \subset Y$ is open. Choose points $y \in Y_{\text{sing}} \setminus R$ and $x \in X$ such that $f(x)=y$. Then, there exists an open neighborhood $U \cong \mathbb{R}^{\dim(X)}$ of $x$ such that $f \vert_{U}$ is a homeomorphism. It follows that $f(U) \cong \mathbb{R}^{\dim(Y)}$ is an open neighborhood of $y \in Y_{\text{sing}}$ which is a contradiction.
\end{proof}
   For a branched covering $f:X \longrightarrow Y$ with a singular target space, the statement of Proposition \ref{ConstantIndex} does not hold for the intrinsic stratification of $Y$ in general. In the following proposition, we aim to refine this stratification so that the cardinality of the fibers over branch points is constant along each connected component of the strata.
    \begin{proposition}\label{RefinedStrat}
    Let $f:X \longrightarrow Y$ be a branched covering of topological spaces such that $X$ is a closed topological manifold and $Y$ is a closed topological pseudomanifold. Furthermore, assume that the branching locus $R \subset Y$ is a closed stratified space of codimension 2 with the intrinsic stratification $R= \bigsqcup_{\beta}T_{\beta}$. Then, for the intrinsic stratification of $Y$, namely $Y= \bigsqcup_{\alpha} S_{\alpha}$, the decomposition 
    \begin{align}\label{refined}
    Y=\big(S_{\alpha_{\text{top}}}\setminus (R \cap S_{\alpha_{\text{top}}}) \big) \bigsqcup_{\alpha, \beta } (S_{\alpha} \cap T_{\beta})
    \end{align}
    is a refined stratification of $Y$, where $S_{\alpha_{\textbf{top}}}$ is the top stratum of $Y$.
    \end{proposition}
    \begin{proof}
    	Let $Y_{\text{sing}} \subset Y$ be the singular locus of $Y$ and $n= \dim(Y)$. From the previous proposition, we know that $Y_{\text{sing}} \subseteq R$. The idea of the proof is to show the intersection $S_{\alpha} \cap T_{\beta}$ is a manifold for all $\alpha$ and $\beta$. Choose a distinguished open neighborhood $U$ of the point $y \in S_{\alpha} \cap T_{\beta}$ in $R$ such that $U \cong \mathbb{R}^{l} \times C(L)$, where $l = \dim(T_{\beta})$ and $L$ is a link of the point $y$. Let $i:R \hookrightarrow Y$ be the inclusion. Let $U^{\prime} \subset Y$ be a distinguished open neighborhood of the point $y$ such that $U^{\prime} \cong \mathbb{R}^{l^{\prime} } \times C(L^{\prime})$, where $l^{\prime}= \dim(S_{\alpha})$ and $L^{\prime}$ is a link of the point $y$ in $Y$.

    	Let $A \ast B$ denote the join of topological spaces $A$ and $B$. Assume $l^{\prime} \leq l$. Note that $U \cong \mathbb{R}^{l} \times C(L) \cong \mathbb{R}^{l^{\prime}} \times C(S^{l-l^{\prime}-1 }) \times C(L) \cong \mathbb{R}^{l^{\prime}} \times C(S^{l-l^{\prime}-1 } \ast L)$, with the conventions that $S^{-1}= \emptyset$ and $\emptyset \ast A=A$, for any topological space $A$. 
    	Note that in Definition \ref{localflatness}, we assumed the existence of the following commutative diagram
    	\begin{align*}		\begin{tikzcd}[ampersand replacement=\&]
    			U^{\prime} \cap S_{\alpha} \arrow[r,  hookrightarrow] \arrow[d] \& U \cap T_{\beta} \arrow[d]   \\
    			\mathbb{R}^{l^{\prime}} \arrow[r, hookrightarrow]  \& \mathbb{R}^{l}  
    		\end{tikzcd}.
    	\end{align*}
    	Such that the inclusion $\mathbb{R}^{l^{\prime}} \hookrightarrow \mathbb{R}^{l}$ is the standard inclusion. Thus, we have the following commutative diagram
    		 \begin{align*}	
    	 \begin{tikzcd}[ampersand replacement=\&]
    	 	U^{\prime} \cap S_{\alpha}(\cong \mathbb{R}^{l^{\prime}})  \arrow[r,  hookrightarrow,] \arrow[d,"\cong"'] \& U (\cong \mathbb{R}^{l^{\prime}} \times C( S^{l-l^{\prime}-1} \ast L)) \arrow[d,hookrightarrow,"i \vert"']   \\
    	 U^{\prime} \cap (S_{\alpha} \cap T_{\beta})  (\cong \mathbb{R}^{l^{\prime}}) \arrow[r, hookrightarrow]  \& U^{\prime}(\cong \mathbb{R}^{l^{\prime}} \times C( L^{\prime}))  \\ 
    	 \end{tikzcd}.
    	 \end{align*}	
    	 Hence, for a point $y \in S_{\alpha} \cap T_{\beta} \subset R$, the open neighborhood $U^{\prime}$ provides a distinguished open neighborhood of $i(y) \in Y$. From the above commutative diagram, it follows that $U^{\prime} \setminus U \cong U^{\prime} \setminus (U^{\prime} \cap R) \cong (\mathbb{R}^{l^{\prime}} \times C(L^{\prime}) \setminus (\mathbb{R}^{l^{\prime}} \times C(S^{l-l^{\prime} -1} \ast L))) \cong \mathbb{R}^{l^{\prime}} \times (L^{\prime} \setminus (S^{l-l^{\prime} -1} \ast L)) \times (0,1] \simeq L^{\prime} \setminus( S^{l-l^{\prime} -1} \ast L)$. Hence, for each $y \in S_{\alpha} \cap T_{\beta}$ the cardinality of the fiber, which is equal to $\operatorname{rk}(f_{\ast}(\underline{\mathbb{Q}}_{X})_{y})$, is constant, as shown in the proof of Proposition \ref{ConstantIndex}. It follows that the map $f \vert_{S_{\alpha} \cap T_{\beta}}$ is an unramified covering and, consequently, $S_{\alpha} \cap T_{\beta}$ and $f^{-1}(S_{\alpha} \cap T_{\beta})$ are topological manifolds.

    	 Assuming $l \leq l^{\prime}$ and using the above notation, we arrive at the following commutative diagram
    	 	 \begin{align*}	
    	 	\begin{tikzcd}[ampersand replacement=\&]
    	 		U \cap T_{\beta}(\cong \mathbb{R}^{l})  \arrow[r,  hookrightarrow,] \arrow[d,"\cong"'] \& U (\cong \mathbb{R}^{l} \times C( L)) \arrow[d,hookrightarrow,"i \vert"']   \\
    	 		U^{\prime} \cap (S_{\alpha} \cap T_{\beta})  (\cong \mathbb{R}^{l}) \arrow[r, hookrightarrow]  \& U^{\prime}(\cong \mathbb{R}^{l} \times C( S^{l^{\prime} -l-1} \ast L^{\prime}))  \\ 
    	 	\end{tikzcd}.
    	 \end{align*}
    	 With the same argument as before, it follows that $S_{\alpha} \cap T_{\beta}$ and $f^{-1}(S_{\alpha} \cap T_{\beta})$ are manifolds. Hence, the claim follows.
         \end{proof}
        Considering the proof of the previous proposition, it follows that the rank of the direct image sheaf $(R^{0}f_{\ast} \underline{\mathbb{Q}}_{X})_{y}$ at each point $y \in R$ is constant along each connected component of the strata in the refined stratification; and consequently, the cardinality of the fiber is likewise constant.

    \begin{corollary}
    In the setting of Proposition~\ref{RefinedStrat}, the cardinality of fibers over branch points is constant on each connected component of the refined stratification.
    \end{corollary}
    \begin{proof}
    Using the same notation as in the proof of Proposition \ref{RefinedStrat}, for a given point $y \in S_{\alpha} \cap T_{\beta}$ and $U^{\prime} \subseteq Y$ a sufficiently small neighborhood of the point $y$, we obtain $U^{\prime} \setminus (U^{\prime} \cap R) \simeq L^{\prime} \setminus (S^{l-l^{\prime} -1} \ast L)$. The claim follows by using the same arguments as in the proof of Proposition \ref{ConstantIndex}. 
    \end{proof}
    Consider the refined stratification $Y=Y_{n} \supset Y_{n-2}(= R) \supset Y_{n-3}  \supset \cdots$ obtained in Proposition \ref{RefinedStrat}. The goal of the following proposition is to show that the pull-back filtration $X \supset f^{-1}(Y_{n-2}) \supset f^{-1}(Y_{n-3}) \supset \cdots$ is in fact a stratification of the closed manifold $X$.
    \begin{proposition}
    	In the setting of Proposition \ref{RefinedStrat}, the decomposition 
    	\begin{align}\label{refinedX}
    		X=f^{-1}\big( S_{\alpha_{\textbf{top}}}\setminus (R \cap S_{\alpha_{ \text{top}  } } ) \big) \bigsqcup_{\alpha, \beta } f^{-1}(S_{\alpha} \cap T_{\beta})
    	\end{align}
    	is a stratification of the closed manifold $X$.
    \end{proposition}
\begin{proof}
	As shown in the proof of Proposition \ref{RefinedStrat}, since the cardinality of the fibers over branch points is constant on each connected component of strata, the restricted maps $f\vert_{f^{-1}(S_{\alpha} \cap T_{\beta}) }$ and $f \vert_{ f^{-1}(S_{\alpha_{\operatorname{top}} } \setminus R)}$ are unramified coverings. As a result, the subspaces $f^{-1}(S_{\alpha} \cap T_{\beta})$ and $f^{-1}(S_{\alpha_{\text{top}}}  \setminus R)$ are submanifolds of $X$. Consider the stratification of $Y$ given by Filtration \ref{refined}. Let $S$ and $T$ be strata of $Y$ such that $S \subset \overline{T}$. Define $A:= f^{-1}(S)$ and $B:=f^{-1}(T)$, and suppose $A \cap \overline{B} \neq \emptyset$. Then for any $x \in A$, we have $f(x) \in S$. Since $S \subset \overline{T}$, every open neighborhood of $f(x)$ intersects $T$. Because $f$ is an open map, every open neighborhood of $x$ intersects $B$. Thus, $x \in \overline{B}$, so $A \subset \overline{B}$. The claim follows as a result.
\end{proof}
\begin{proposition}\label{SingDecompProp1}
	Let $f:X \longrightarrow Y$ be a branched covering with $X$ a  closed topological manifold of dimension $n$, $Y$ a closed topological pseudomanifold, and the branch locus $R \subset Y$ a closed topological stratified space of codimension 2. Consider the refined stratification $Y=Y_{n} \supset Y_{n-2}(=R) \supset Y_{n-3} \supset Y_{n-4} \supset \dots$ obtained in Proposition \ref{RefinedStrat} and the inclusions $j_{k}: U_{k} \hookrightarrow U_{k+1}$ for $U_{k}=Y \setminus Y_{n-k}$. Furthermore, let $X_{i}:=f^{-1}(Y_{i})$, and consider the inclusions $i_{k}:V_{k} \hookrightarrow V_{k+1}$ for $V_{k}=X \setminus X_{n-k}$. Then in the derived category $D^{b}_{c}(Y)$, the following isomorphism holds:
	\begin{align}\label{SingDecomp1}
		&	Rf_{\ast}\tau_{\leq \overline{p}(n)-n} Ri_{n \ast} \cdots 	\tau_{\leq \overline{p}(3)-n} Ri_{3 \ast} \tau_{\leq \overline{p}(2)-n} Ri_{2 \ast} 
        \underline{\mathbb{Q}}_{X \setminus f^{-1}(R)}[n] \simeq  \\ & \tau_{\leq\overline{p}(n)-n} Rj_{n \ast} \cdots 	\tau_{\leq \overline{p}(3)-n} Rj_{3 \ast} \tau_{\leq \overline{p}(2)-n} Rj_{2 \ast} \underline{\mathbb{Q}}_{Y \setminus R}[n] \nonumber \\ \oplus & \tau_{\leq\overline{p}(n)-n} Rj_{n \ast} \cdots 	\tau_{\leq \overline{p}(3)-n} Rj_{3 \ast} \tau_{\leq \overline{p}(2)-n} Rj_{2 \ast}\mathcal{L}_{Y \setminus R}[n] \nonumber,
	\end{align}
	where $\mathcal{L}$ is the local system defined in Equation \ref{Decomp1}, associated with the unbranched covering $f\vert_{X \setminus f^{-1}(R)}: X \setminus f^{-1}(R) \longrightarrow Y \setminus R$ and $\overline{p}$ is either the lower or the upper middle perversity.
\end{proposition}
\begin{proof}
Consider the commutative diagram \begin{align*}		\begin{tikzcd}[ampersand replacement=\&]
		X \setminus f^{-1}(R) \arrow[r,  hookrightarrow,"i_{2}"] \arrow[d, "f \vert_{2}"'] \& X \setminus X_{n-3}  \arrow[d,"f \vert_{3}"] \arrow[r, hookrightarrow,"i_{3}"] \& X \setminus X_{n-4}  \arrow[d,"f \vert_{4}"] \arrow[r, hookrightarrow,"i_{4}"]  \& \cdots \\
		Y \setminus R \arrow[r, hookrightarrow,"j_{2}"]  \& Y \setminus Y_{n-3} \arrow[r,hookrightarrow,"j_{3}"] \& Y \setminus Y_{n-4} \arrow[r, hookrightarrow,"j_{4}"] \& \cdots  \\ 
	\end{tikzcd}.
\end{align*}
	Proposition \ref{PropDecomp1} implies that $Rf \vert_{\ast} \underline{\mathbb{Q}}_{X \setminus f^{-1}(R)} \simeq \underline{\mathbb{Q}}_{Y \setminus R} \oplus \mathcal{L}$, where the local system $\mathcal{L}$ is defined in Equation \ref{Decomp1}. Let $\overline{p}$ be either the lower or the upper middle perversity, and let $\tau_{\leq}$ be the truncation functor. Now, we consider the operator
	\begin{align*}
		\tau_{\leq \overline{p}(n)-n} Rj_{n \ast} \cdots 	\tau_{\leq \overline{p}(3)-n} Rj_{3 \ast} \tau_{\leq \overline{p}(2)-n} Rj_{2 \ast} .
	\end{align*}
	Applying this operator to both sides of the previous isomorphism, we obtain \begin{align*}
		& \tau_{\leq \overline{p}(n)-n} Rj_{n \ast} \cdots 	\tau_{\leq \overline{p}(3)-n} Rj_{3 \ast} \tau_{\leq \overline{p}(2)-n} Rj_{2 \ast} 
		Rf \vert_{\ast}\underline{\mathbb{Q}}_{X \setminus B}[n] \simeq \\  &\tau_{\leq\overline{p}(n)-n} Rj_{n \ast} \cdots 	\tau_{\leq \overline{p}(3)-n} Rj_{3 \ast} \tau_{\leq \overline{p}(2)-n} Rj_{2 \ast} \underline{\mathbb{Q}}_{Y \setminus R}[n] \\ \oplus & \tau_{\leq\overline{p}(n)-n} Rj_{n \ast} \cdots 	\tau_{\leq \overline{p}(3)-n} Rj_{3 \ast} \tau_{\leq \overline{p}(2)-n} Rj_{2 \ast}\mathcal{L}_{Y \setminus R}[n].
	\end{align*}
	In the next step, similar to the proof of Theorem \ref{MainTheo1}, we want to show that for each point $y \in R$ there exists a sufficiently small neighborhood $U^{\prime} \subseteq Y$ of $y$ such that for every open neighborhood $W \subseteq U^{\prime}$ of the point $y$, each connected component of the preimage $f^{-1}(W)$ is contractible. Let $U^{\prime} (\cong \mathbb{R}^{l^{\prime}} \times C(L^{\prime})) \subseteq Y$ and $U(\cong \mathbb{R}^{l} \times C(L)) \subseteq R$ be open neighborhoods of the point $y \in S_{\alpha} \cap T_{\beta}$, where $l^{\prime}=\dim(S_{\alpha} \cap T_{\beta})$, $l= \dim(T_{\beta})$, and $L^{\prime}$ and $L$ are links of the point $y$ in $Y$ and $R$, respectively. As shown in the proof of Proposition \ref{RefinedStrat}, for $U^{\prime}$ sufficiently small, we have $U^{\prime} \setminus U \cong U^{\prime} \setminus (U^{\prime} \cap R) \cong \mathbb{R}^{l^{\prime}} \times (L^{\prime} \setminus (S^{l-l^{\prime} -1} \ast L)) \times (0,1]$. Let $V \subseteq f^{-1}(U^{\prime})$ be a connected component. It follows that $f\vert_{V}:V \longrightarrow U^{\prime}$ is a branched covering with the cardinality of fibers along $(S_{\alpha} \cap T_{\beta}) \cap U^{\prime}$ equal to one. Note that $\partial \overline{U^{\prime}} \setminus ( \partial (\overline{U^{\prime}}) \cap R) \cong D^{l^{\prime}} \times (L^{\prime} \setminus (S^{l-l^{\prime} -1} \ast L )) \times  \{1\}$. It follows that the restriction map $f\vert_{ V \cap f^{-1}((\partial \overline{U^{\prime}} \setminus ( \partial (\overline{U^{\prime}}) \cap R))) } := f \vert_{\partial}$ is an unramified covering. On the other hand, note that the quotient map $f \vert_{ \operatorname{quo.} }: \sfrac{V}{(V \cap f^{-1}(U^{\prime} \cap R))} \longrightarrow \sfrac{U^{\prime}}{(U^{\prime} \cap R)}$ is a branched covering with branch locus being a single point, since $V$ is a connected component of $f^{-1}(U^{\prime})$, which implies that the restriction map $f \vert_{V \cap f^{-1}(U^{\prime} \cap (S_{\alpha} \cap T_{\beta} ))}$ is a homeomorphism. As a result we have a homeomorphism between the map $C(f \vert_{\partial}):=C(f\vert_{ V \cap f^{-1}((\partial \overline{U^{\prime}} \setminus ( \partial (\overline{U^{\prime}}) \cap R))) })$ and the map $ f\vert_{\operatorname{quo.}}$. Note that since $U^{\prime} \cap R \simeq \ast$ and $U^{\prime} \cap R \hookrightarrow U^{\prime}$ is a cofibration; hence the previous homeomorphism of maps implies $V \simeq \sfrac{V}{V \cap f^{-1}(U^{\prime} \cap R)} \cong C\big(V \cap f^{-1}((\partial \overline{U^{\prime}} \setminus ( \partial (\overline{U^{\prime}}) \cap R )))\big) \simeq \ast$.

	Similar to Theorem \ref{MainTheo1}, Consider the adjunction between the derived direct image $Rf\vert_{3\ast}$ and the derived inverse image $Lf\vert_{3}^{\ast}$, together with the fact that $Lf\vert_{3}^{\ast}$ commutes with truncation functors (since it is exact). Hence, for the truncation $\tau_{\leq -n}$, the counit of the adjunction gives a natural morphism
	
	\begin{align*}
		\tau_{\leq \overline{p}(2) -n} Rf\vert_{3\ast} Ri_{2\ast} \underline{\mathbb{Q}}_{X \setminus B}[n] \longrightarrow Rf \vert_{3\ast} \tau_{\leq \overline{p}(2)-n} Ri_{2\ast} \underline{\mathbb{Q}}_{X \setminus B}[n]
	\end{align*}
	
	Note that $\tau_{\leq \overline{p}(2) -n} Ri_{2\ast} \underline{\mathbb{Q}}_{X \setminus B}[n] \simeq \underline{\mathbb{Q}}_{X \setminus B_{n-3}}[n]$ since $X$ is a manifold. The existence of sufficiently small contractible open neighborhoods, as constructed above, such that the preimage under the map $f$ is contractible as well, implies that $R^{l}f_{\ast}(\tau_{\leq \overline{p}(2) -n} Ri_{2\ast} \underline{\mathbb{Q}}_{X \setminus B}[n])_{y}=0$ for $l > 0$ and $y \in Y$. For $l=0$, since $f$ is a finite map, the above natural morphism induces an isomorphism on the stalks. It follows that the induced morphism in the derived category is a quasi-isomorphism. Thus, we have the quasi-isomorphism
	\begin{align*}
		\tau_{\leq \overline{p}(2)-n} Rf \vert_{3 \ast} Ri_{2 \ast} \underline{\mathbb{Q}}_{X \setminus B}[n] \simeq  Rf \vert_{3 \ast} \tau_{\leq \overline{p}(2)-n} Ri_{2 \ast} \underline{\mathbb{Q}}_{X \setminus B}[n].
	\end{align*}

	 The claim follows by repeating the previous step inductively for each square in the commutative diagram above.
\end{proof}

Although Deligne's sheaf complex for a constant sheaf on a top stratum is independent of the stratification (up to quasi-isomorphism), the same is not true for a non-trivial local system $\mathcal{L}$. The stratification with respect to which the complex $IC_{Y}(\mathcal{L})$ is constructed must therefore be specified. To be more precise, let $U_{2}$ and $U_{2}^{\prime}$ be the top strata associated with two different stratifications. Without loss of generality, assume that $U_{2} \subsetneq U_{2}^{\prime}$. There is no canonical way to extend a non-trivial local system $\mathcal{L}$ on $U_{2}$ to a local system $\mathcal{L}^{\prime}$ on $U_{2}^{\prime}$ such that $\mathcal{L}^{\prime} \vert_{U_{2}} \simeq \mathcal{L}$. In fact, the direct image of $\mathcal{L}$ under the inclusion $U_{2} \hookrightarrow U_{2}^{\prime}$ is generally not a local system. Using Proposition \ref{RefinedStrat}, the sheaf complexes appearing in Isomorphism \ref{SingDecomp1} can be interpreted as Deligne's sheaf complexes.
\begin{theorem}\label{MainTheo2}
	Let $f:X \longrightarrow Y$ be a branched covering, where $X$ is a closed topological manifold of dimension $n$, the topological space $Y$ is a closed topological pseudomanifold, and the branch locus is a closed topologically stratified subspace $R \subset Y$ of codimension 2. We consider the refined stratification obtained in Proposition \ref{RefinedStrat}. Then, in $D^{b}_{c}(Y)$, the following isomorphism holds:
	\begin{align}
			Rf_{\ast}{\underline{\mathbb{Q}}}_{X}[n] \simeq IC_{Y}(\underline{\mathbb{Q}}_{Y \setminus R}) \oplus IC_{Y}(\mathcal{L}_{Y \setminus R}),
	\end{align}
	where $\mathcal{L}_{Y \setminus R}$ is the local system defined in Equation \ref{Decomp1}, associated with the unbranched covering $f\vert_{X \setminus f^{-1}(R)}: X \setminus f^{-1}(R) \longrightarrow Y \setminus R$.
    \end{theorem} 
    \begin{proof}
    Using the pull-back stratification of $X$ given by Filtration \ref{refinedX}, since $X$ is a topological manifold, the left-hand side of Isomorphism \ref{SingDecomp1} is quasi-isomorphic to $Rf_{\ast}\underline{\mathbb{Q}}_{X}[n]$. The first term on the right-hand side, up to a shift, is Deligne's sheaf for the constant sheaf with stalks $\mathbb{Q}$ on the top stratum and is therefore (up to quasi-isomorphism) independent of the chosen stratification. However, since the local system $\mathcal{L}_{Y \setminus R}$ is in general not trivial, we must specify the stratification with respect to which Deligne's sheaf complex $IC_{Y}(\mathcal{L}_{Y \setminus R})$ is constructed; this is Stratification \ref{refined}.
    \end{proof}

    	Let $\bar{m}$ and $\bar{n}$ be the lower and the upper middle perversities, respectively. For an arbitrary topological pseudomanifold $Y$ with strata of odd codimension, Deligne's sheaf complexes $IC_{Y}^{\bar{m}}(\mathbb{Q})$ and $IC_{Y}^{\bar{n}}(\mathbb{Q})$ are not quasi-isomorphic in general. In fact, Verdier duality states that if $Y$ is an orientable topological pseudomanifold of pure dimension $k$, then the Verdier dual of Deligne's sheaf complex associated to the perversity $\bar{m}$ and to the constant sheaf on the non-singular part of the pseudomanifold is quasi-isomorphic to Deligne's sheaf complex associated to the perversity $\bar{n}$ and to the constant sheaf on the non-singular part of the pseudomanifold, up to a shift by $k$. However, under the assumptions of the above theorem, it follows that the complexes of sheaves $IC_{Y}^{\bar{n}}(\mathbb{Q}) \oplus IC_{Y}^{\bar{n}}(\mathcal{L}_{Y \setminus R})$ and $IC_{Y}^{\bar{m}}(\mathbb{Q}) \oplus IC_{Y}^{\bar{m}}(\mathcal{L}_{Y \setminus R})$ are quasi-isomorphic. This corollary, which is due to the possible presence of strata of odd codimension, is unique to the topological setting and has no algebraic counterpart.

\end{document}